\documentclass[11pt,preprint,reqno]{amsart} 
\usepackage{amsmath,amssymb,amsthm}
\usepackage{mathscinet} 
\usepackage[usenames,dvipsnames]{xcolor}
\usepackage[mathscr]{eucal}

\usepackage[latin1]{inputenc}

\usepackage{syntonly}  

\usepackage{dsfont} 
\usepackage{graphicx}
\usepackage{xcolor}
\usepackage{enumitem}

\newtheorem{theorem}{Theorem}

\newtheorem{proposition}[theorem]{Proposition}
\newtheorem{definition}[theorem]{Definition}

\newtheorem{lemma}[theorem]{Lemma}

\newtheorem{remark}{Remark}

\numberwithin{equation}{section}
\def\qed{\hbox{${\vcenter{\vbox{		 
   \hrule height 0.4pt\hbox{\vrule width 0.4pt height 6pt
   \kern5pt\vrule width 0.4pt}\hrule height 0.4pt}}}$}}

\def\cF{\mathcal F}

\def\cH{\mathcal H}

\def\bD{\mathbb D}
\def\bE{\mathbb E}

\def\bR{\mathbb R}

\def\uh{\underline{h}}
\def\uu{\underline{u}}
\def\uz{\underline{z}}

\newcommand{\commentout}[1]{}

\def\uu{\underline u}

\newcommand{\comment}[1]{$\null$}

\begin{document}

\title[\emph{Absolute continuity of the law for the 2D stochastic Navier-Stokes equations.
}]{\emph{Absolute continuity of the law for the two dimensional stochastic Navier-Stokes equations }}

\author{Benedetta Ferrario, Margherita Zanella}
\address{Benedetta Ferrario, Margherita Zanella \newline Università di Pavia, Dipartimento di Matematica "F. Casorati", via Ferrata 5, 27100 Pavia, Italy}
\email{benedetta.ferrario@unipv.it, margherita.zanella01@ateneopv.it}


\begin{abstract}
We consider the two dimensional  Navier-Stokes equations in
vorticity form with a stochastic forcing term given by a 
gaussian noise, white in time and coloured in space.
First, we prove existence and uniqueness of a weak (in the Walsh sense) solution process $\xi$ and we show that, 
if the initial vorticity $\xi_0$ is continuous in space, then there exists a  space-time continuous 
version of the solution.
In addition we show that the solution $\xi(t,x)$ (evaluated at fixed points in time and space) 
is locally differentiable in the Malliavin calculus sense and that its image law
 is absolutely continuous with respect to the Lebesgue measure on $\mathbb{R}$. 
\end{abstract}

\maketitle

\def\uh{\underline{h}}
\def\uu{\underline{u}}
\def\uz{\underline{z}}

\section{Introduction}
The analysis of stochastic partial differential equations concerns problems of existence, 
uniqueness and properties of the solution processes.
In particular, there has been a lot of activity in the last years studying the regularity in the Malliavin 
sense for solutions to stochastic partial differential equations. 
Here we are interested in looking for the existence of a density for the law of the random variable 
given by the  solution process at  fixed points in time and space. This property is important in the analysis
of hitting probabilities (see \cite{DKN2009,DS2010}) and concentration inequalities (see \cite{NV2009}).
Most of the  literature on this subject concerns the heat and wave equations 
(see e.g. \cite{Nualart2006}, \cite{Bally1998}, \cite{Mueller2008},
\cite{Sanz_Sardanyons}, \cite{Marquez-Carreras2001}, \cite{Marinelli2013} and the references therein). 
Moreover, there is a paper dealing with  the Cahn-Hilliard equation
(see \cite{Cardon-Weber2001})
and some papers dealing with the one dimensional Burgers equation (see \cite{Morien1999,NualartZaidi1997}).
Our aim is to deal with  stochastic fluid dynamical equation in
dimension bigger than one. As we shall see in Section \ref{exist_uniq_cont}, our equation can be written as a stochastic 
parabolic nonlinear equation in a two dimensional spatial domain with a nonlinear term which is of a form different from that 
studied in other papers about stochastic 
parabolic nonlinear equations in spatial dimension bigger than 1 (see \cite{Marquez-Carreras2001}, \cite{Marinelli2013}).

Therefore, we consider the 
two dimensional stochastic Navier-Stokes equations 
\begin{equation}
\label{NS}
\begin{cases}
\dfrac{\partial v}{\partial t}(t,x)+ (v (t,x)\cdot \nabla)  v(t,x)- \nu \Delta v(t,x)+ \nabla p(t,x)= n(t,x)& (t,x) \in \left[0,T\right] \times D
\\
\nabla \cdot  v(t,x)=0 & (t,x) \in \left[0,T\right] \times D
\\
v(0,x)= v_0(x) & x \in D.
\end{cases}
\end{equation}
describing the motion  of a viscous incompressible fluid in a domain
$D\subset \mathbb R^2$. The  unknowns are the velocity vector  $v$ and
the pressure $p$, whereas the data are the viscosity  $\nu>0$, the initial
velocity $v_0$ and the  stochastic forcing term $n$. Suitable boundary
conditions are associated to system \eqref{NS}; here we choose to work
on the torus, so $D=\left[0, 2 \pi\right]^2$ and  periodic boundary
conditions are assumed.

Taking formally the \textit{curl} of both sides of the first equation
in \eqref{NS} we get the vorticity formulation, where the unknown is
the vorticity $\xi=\nabla^\perp \cdot  v\equiv 
\partial_{x_1} v_2 -\partial_{x_2} v_1$. Indeed, 
 the \textit{curl} of a planar vector filed is a
vector orthogonal to the plane, hence with only one significant
component $\xi$. Therefore, for regular enough solutions, 
system \eqref{NS} is
equivalent to
\begin{equation}
\label{vort_0}
\begin{cases}
\displaystyle  \frac{\partial\xi}{\partial t}(t,x)- \nu \Delta \xi(t,x)+v(t,x) \cdot \nabla \xi(t,x)
       = w({\rm d}x,{\rm d}t)&\quad (t,x) \in [0,T] \times D
\\ 
\nabla \cdot  v(t,x)=0 &\quad (t,x) \in [0,T] \times D
\\ 
\xi(t,x) = \nabla^{\perp} \cdot v(t,x)&\quad (t,x) \in [0,T] \times D
\\
\displaystyle \xi(0,x)=\xi_0(x) &\quad x \in D
\end{cases}
\end{equation}
with periodic boundary conditions.  For simplicity we put $\nu=1$ from
now on.
The random force $w$ acting on the system is formally equal to the \textit{curl} of $n$ 
appearing in \eqref{NS}; $w({\rm d}x, {\rm d}t)$ is the formal notation for some Gaussian 
perturbation defined on some probability space (for the details see Subsection \ref{hyp_res}).
We shall see that system  \eqref{vort_0} can be rewritten as  a closed
equation for the vorticity, since 
$v$ can be explicitly expressed in terms of $\xi$ by means of the
Biot-Savart law $v=k\ast\xi$
(see Subsection \ref{biot_savart_section}).

We interpret Eq. \eqref{vort_0} in the sense of Walsh (see \cite{Walsh1986}). Let $g(t,x,y)$ be the 
fundamental solution to the heat equation on the flat torus (see
Subsection \ref{sec3}).
We shall see that a random field $\xi=\{\xi(t,x), t \in \left[0,T\right] \times D\}$ 
is a solution to equation \eqref{vort_0} if it satisfies the evolution equation
\begin{multline}
\label{Walsh_0}
\xi(t,x)=\int_D g(t,x,y)\xi_0(y)\, {\rm d}y 
+ \int_0^t \int_D \nabla_yg(t-s, x,y) \cdot  v(s,y) \xi(s,y) \, {\rm d}y \,{\rm d}s
\\
+
\int_0^t \int_D g(t-s,x,y)\,  w({\rm d}y, {\rm d}s)\qquad\qquad
\end{multline}
with $v=k\ast\xi$. 
The stochastic integral
will be explicitly defined in Subsection \ref{hyp_res}.
Notice that in the present work  we consider an additive noise, 
that is the stochastic forcing term is independent of 
the unknown process $\xi$. This particular choice is made only in
order to highlight the novelties of the results when compared to the one dimensional Burgers equation.
Nevertheless, with standard techniques it is possible to extend the
results to the multiplicative case and this shall be the object of a subsequent paper.

In the first part of the paper we shall prove the existence and uniqueness
of the solution to problem \eqref{vort_0}. We follow an approach similar to
\cite{GyongyNualart1999} for the one dimensional stochastic Burgers equation and to
\cite{Cardon-Weber2001} for the Cahn-Hilliard stochastic equation. 
The regularization property of the heat kernel as stated in Lemma \ref{regularization} 
plays a key role in our method. Since the non linear term that appears
in \eqref{Walsh_0} is non Lipschitz, we adopt a method of localization: 
by means of a contraction principle, we prove at first the result for the smoothed 
equation with truncated coefficient. This kind of result provides the uniqueness 
for the solution to \eqref{vort_0} and its local existence, namely the existence 
on the time interval $\left[0,\tau\right]$ where $\tau$ is a stopping time. 
To prove the global existence we show that $\tau=T$ $\mathbb{P}$-a.s.  
We then study the regularity of $\xi$ proving that if $\xi_0$ is a
continuous 
function on $D$, then the solution admits a modification which is a space-time continuous process.

In the second part of the paper we study the regularity of the solution in 
the sense of stochastic calculus of variations, namely we prove the existence of 
the density of the random variable $\xi(t,x)$, for fixed 
$(t, x) \in \left[0,T\right] \times D$. For this we use the Malliavin 
calculus (see \cite{Nualart2006}) associated to the noise that appears 
in \eqref{vort_0}. We prove at first that for any fixed $(t,x) \in
\left[0,T\right] \times D$ the random variable $\xi(t,x)$ belongs to the Sobolev 
space $\mathbb{D}_{loc}^{1,p}$ for every $p>4$.
Then we prove that the law of  $\xi(t,x)$ is absolutely continuous with 
respect to the Lebesgue measure on $\mathbb{R}$. We point out here that the 
localization argument we use in order to achieve this result does 
not provide the smoothness of the density since we do not have the boundedness 
of the derivatives of every order. Moreover, let us notice that 
the technique of  analysis of the existence of the density  by means of Malliavin
calculus  is suited for a scalar unknown; the case for a vector unknown is
much more involved (see, e.g., \cite{Nualart2006}).
This is the reason why we  work on the Navier-Stokes equations in vorticity
form \eqref{vort_0} instead of the usual formulation \eqref{NS} with respect to the vector
velocity.

The main results of the paper are the following.
\begin{theorem}
\label{existence_thm}
Let $b>0$ in \eqref{delta} and $p>2$. 
If $\xi_0 \in L^p(D)$, 
then there exists a unique $\mathcal{F}_t$-adapted solution to 
equation \eqref{Walsh_0} which is continuous with values in $L^p(D)$.
Moreover, if $\xi_0 \in C(D)$ the solution admits a modification which is a space-time continuous process.
\end{theorem}

\begin{theorem}
\label{density_thm}
Let $b>1$ in \eqref{delta}. If $\xi_0 \in C(D)$, 
then for every $t\in \left[0,T\right]$ and $x \in D$ the image law 
of the random variable $\xi(t,x)$ is absolutely continuous w.r.t. to the Lebesgue measure on $\mathbb{R}$.
\end{theorem}

The paper is organized as follows: in Section \ref{mathematical_setting} we define the functional spaces, we state the results concerning the needed estimates of heat kernel on the flat torus, we present the Biot-Savart law that exploit the relation between the velocity and the vorticity and we state the hypothesis concerning the random forcing term. In Section \ref{prliminary_lemmas} we present some technical lemmas. In Section \ref{exist_uniq_cont} we establish the existence and uniqueness of the solution to \eqref{vort_0} as well as its $\mathbb{P}$-a.s. space-time continuity. In Section \ref{Malliavin_section} we prove the absolute continuity of the solution $\xi(t,x)$, for $t \in \left[0,T\right]$, $x \in D$. Finally, the estimates of the heat kernel and its gradient are proved in \ref{Appendix}.

\smallskip

{\small Notation. In the sequel, we shall indicate with $C$ a constant that may varies from line to line. In certain cases, we write $C_{\alpha,\beta,\dots}$ to 
emphasize the dependence of the constant on the parameters $\alpha, \beta, \dots$.}

\section{Mathematical Setting }
\label{mathematical_setting}

\subsection{Spaces and operators.}
We denote by $ x=(x_1,x_2)$ a generic point of $\mathbb{R}^2$ and by
$$x \cdot y= x_1y_1+x_2y_2 \qquad \text{and} \qquad |x|=\sqrt{x \cdot x},\qquad  x,  y \in \mathbb{R}^2$$
the scalar product and the norm in $\mathbb{R}^2$.
Given $z=\mathcal{R}z+i \mathcal{I}z \in \mathbb{C}$ we denote by $|z|$ 
its absolute value and by $\bar z$ its complex coniugate: 
$|z|=\sqrt{(\mathcal{R}z)^2+(\mathcal{I}z)^2}$, $\bar z=\mathcal{R}z-i \mathcal{I}z$.
We define $\mathbb{Z}_+^2=\{k=(k_1,k_2)\in \mathbb{Z}^2:k_1 >0\} \cup \{
k=(0,k_2)\in \mathbb{Z}^2:k_2>0\}$ and $\mathbb{Z}_0^2=\mathbb{Z}^2 \setminus \{ 0\}$.

Let  $D=\left[0,2 \pi\right]^2$, we consider the space $L^2_{\sharp}(D)$ of all complex-valued 
$2 \pi$-periodic functions in $x_1$ and $x_2$ which are measurable and square integrable on $D$, endowed with the scalar product 
\begin{equation*}
\langle f,g \rangle_{L^2(D)}=\int_D f(x) \overline{g(x)}\, {\rm d}x
\end{equation*}
and the norm $\| \cdot\|_{L^2(D)}=\sqrt{\langle \cdot, \cdot\rangle_{L^2(D)}}$.
We also consider the space $\left[L^2_{\sharp}(D)\right]^2$ consisting of all pairs $u=(u_1, u_2)$ of complex-valued periodic functions 
endowed with the inner product
\begin{multline*}
\langle u, v \rangle_{\left[L^2(D)\right]^2}:=\int_D  u(x) \cdot  \overline{v(x)}\, {\rm d}x 
\\
= \int_D \left[u_1(x)\overline{v_1(x)}+u_2(x)\overline{v_2(x)}\right]\, {\rm d}x, \qquad  u, v \in
\left[L^2_{\sharp}(D)\right]^2.
 \end{multline*}
 An orthonormal basis for the space $L^2_{\sharp}(D)$ is given by $\{e_k\}_{k \in \mathbb{Z}^2}$, where
\begin{equation}
\label{base}
e_k(x)=\frac{1}{2\pi}e^{ i k \cdot x}, \qquad x \in D, \ k \in \mathbb{Z}^2.
\end{equation}
As usual in the periodic case, we deal with mean value zero  vectors.
This gives a simplification in the mathematical treatment but  does not prevent to consider 
non zero mean value vectors: this can be dealt in a similar way (see \cite {Temam1995}). 
We use the notation $\dot L^2_{\sharp}(D)$ to keep tracks of the zero-mean condition.
An orthonormal system for the space 
$\dot L^2_{\sharp}(D)$, formed by eigenfunctions of the operator $-\Delta$ with associated eigenvalues $\lambda_k=|k|^2$, is given by $\{e_k\}_{k \in \mathbb{Z}^2_0}$ with $e_k$ as in \eqref{base}.
The real-valued functions in $\dot L^2_{\sharp}(D)$ can be characterized by their Fourier series expansion as follows
\begin{equation*}
\dot L^2_{\sharp}(D)=
\{f(x)= \sum_{k\in \mathbb{Z}^2_0} f_k e_k(x): \bar f_k=f_{-k}\text{ for any } k,
\sum_{k\in \mathbb{Z}^2_0}|f_k|^2 < \infty\}.
\end{equation*}
For every $p>2$, with $\dot L_{\sharp}^p(D)$ we denote the subspaces of $L^p(D)$ 
consisting of zero mean and periodic scalar functions. These are Banach spaces with norms 
inherited from $L^p(D)$.

Let $A$ denote the Laplacian operator $-\Delta$ with periodic boundary conditions. 
For every $b \in \mathbb{R}$, we define the powers of the operator $A$ as follows: 
\begin{equation*}
\text{if } 
f= \sum_{k \in \mathbb{Z}^2_0}f_k e_k \qquad 
\text{then }
A^b f= \sum_{k \in \mathbb{Z}^2_0} |k|^{2b} f_k e_k
\end{equation*}
and 
\[
D(A^b)=\{f= \sum_{k \in \mathbb{Z}^2_0}f_k e_k: 
\sum_{k \in \mathbb{Z}^2_0} |k|^{4b}|f_k|^2<\infty\}.
\]

For any $b\in \mathbb{R}_+$ and $p\ge 1$ 
we set  
\[
W^{b,p}(D)=\{ f \in \dot L_{\sharp}^p(D): A^{\frac b2} f \in
\dot L_{\sharp}^p(D)\}.
\]
These are Banach spaces with the usual norm; when $p=2$ they become
Hilbert spaces and we denote them by $W^b$.
For $b<0$ we define $W^b$ as the dual space of $W^{-b}$ with respect to
the $L^2$-scalar product.

Similarly, we proceed to define the space regularity of vector fields which are periodic, zero mean value and divergence free. We have the corresponding action of the Laplace operator on each component of the vector. 
Therefore we define the space 
\[
H=\{ v \in [\dot L^2_{\sharp}(D)]^2: \nabla\cdot v=0\}
\]
where the divergence free condition has to be understood in the distributional sense. 
This is an Hilbert space with the scalar product inherited from $\left[L^2(D)\right]^2$.
We denote the norm in this space by $|\cdot|_H$, $| u|^2_H:= \langle u,  u \rangle_H$.
A basis for the space 
$H$ is $\{\frac{k^{\perp}}{|k|} e_k\}_{k \in \mathbb{Z}^2_0}$, 
where $k^{\perp}=(-k_2,k_1)$ and $e_k$ is given in \eqref{base}.
For $p>2$ let us set $L_p(D):= H \cap \left[L^p(D)\right]^2$. These
are Banach spaces with norms inherited from $\left[L^p(D)\right]^2$. 
Similarly, for vector spaces we set 
\[
H^b_p(D)= \{v \in L_p(D): A^{\frac b2} v \in L_p(D)\}.
\]
These are Banach spaces with the usual norm; when $p=2$ they become Hilbert spaces and we denote them by $H^b$.
For $b<0$ we define $H^b$ as the dual space of $H^{-b}$ with respect to the $H$-scalar product.

The Poincar\'e inequality holds; moreover, the zero mean value
assumption provides that $\|v\|_{H^b_p(D)}$ is equivalent to 
$\left(\| v\|^p_{L_p(D)}+\| v\|^p_{H^b_p(D)}\right)^{\frac 1p}$.

In the sequel we shall use the Sobolev embedding Theorem (see for instance \cite[Theorem 9.16]{brezis2010functional}):
\begin{itemize}
\item for every $2<p< \infty$ the space $H^1_p(D)$ is compactly embedded in $L_{\infty}(D)$, namely there exists a constant $C$ (depending on $p$ such that):
\begin{align}
\label{Sobolev}
\| v\|_{L_{\infty}(D)} 
\le  C\| v\|_{H^1_p(D)},
\end{align}
\item the space $W^a(D)$
is compactly embedded in $L^{\infty}(D)$ for $a>1$.
\end{itemize}
\smallskip

{\small Notation. In the sequel, spaces over the domain $D$ will be denoted without explicitly mentioning the domain, e.g. $L^p$ stands for $L^p(D)$. By an innocuous abuse of notation, the scalar product $\langle \cdot,\cdot \rangle_{\left[L^2\right]^2}$ will be denoted by $\langle \cdot,\cdot \rangle_{L^2}$ and the norm $\|\cdot\|_{\left[L^p\right]^2}$ by $\|\cdot\|_{L^p}$.

Given two normed vector spaces $(U, \|\cdot\|_U)$ and $(V, \|\cdot\|_V)$, by $\mathcal{L}(U,V)$ we denote the space of all linear bounded operators from $U$ into $V$. We write $\langle \cdot, \cdot \rangle$ for the scalar product $\langle \cdot, \cdot \rangle_{U^{\prime} \times U}$ in the duality $U^{\prime}$, $U$.}

\subsection{The Heat Kernel}
\label{sec3}
We deal with the heat kernel $g$ appearing in equation \eqref{Walsh_0}:
we need suitable estimates on $g$ since its 
regularizing effect (see Lemma \ref{regularization}) will play a key role. 

The operator $-A$ generates a semigroup $S(t)=e^{-tA}$: 
for $\xi \in \dot L^2_{\sharp}$  and $t \in \left[0,T\right]$ we have
\begin{equation}
\label{semigroup}
\left[S(t)\xi\right](x)
=\sum_{k \in \mathbb{Z}^2} e^{-|k|^2t} \langle \xi, e_k \rangle_{L^2} e_k(x)
=\frac{1}{2 \pi}\sum_{k \in \mathbb{Z}^2} \langle \xi, e_k \rangle_{L^2} e^{-t|k|^2+ik \cdot x}.
\end{equation}
Moreover, the action of the semigroup on the function $\xi$ can be expressed as the convolution 
\begin{equation}
\label{kernel}
\left[S(t)\xi\right](x)
= \int_Dg(t, x, y) \xi(y) \, {\rm d}y
\end{equation}
where $g$ is the fundamental solution (or heat kernel) to the problem 
\begin{equation}
\label{Heat1}
\begin{cases}
\frac{\partial}{\partial t}u(t,x)- \Delta  u(t,x)=0, & (t,x) \in \left(0,T\right] \times D
\\
u(t, \cdot) \ \text{is periodic}, &  t \in \left[0,T\right]
\\
u(0,x)=\delta_0(x-y), & x,y \in D.
\end{cases}
\end{equation} 
By means of Fourier series expansion we recover 
\begin{equation}
\label{kernel_fourier}
g(t,x,y) 
=\frac{1}{(2 \pi)^2} \sum_{k \in \mathbb{Z}^2} e^{-t|k|^2+ik\cdot (x-y)}.
\end{equation}
We shall need another expression of the kernel obtained by means of
the method of images 
(for more details see for instance \cite[Chapters 2.7$\S$5 and 2.11$\S$3]{dym1972fourier} and \cite[Chapter 7$\S$2]{steinintroduction}):
\begin{equation}
\label{images}
g(t,x,y)
=\frac{1}{4\pi t} \sum_{k \in \mathbb{Z}^2}e^{-\frac{|x-y+2 k \pi|^2}{4t}}.
\end{equation}

It is easy, using \eqref{kernel_fourier} or \eqref{images},  
to check the following properties
\begin{proposition} \label{regol_g}
For any  $x,y \in D$ and $t>0$ we have
\begin{itemize}
\item Symmetry: $g(t,x,y)=g(t,y,x)$,
\item $g(t,x,y)=g(t,0,x-y)$.
\end{itemize}
\end{proposition}
Following an idea of \cite{Morien1999}, we obtain  estimates  on the
heat kernel and its gradient in the two dimensional case.
\begin{theorem}
\label {beta_gradient}
For fixed $0<s<t$ and $x \in[0,2\pi]$ the following estimates hold:
\begin{enumerate}[label=\roman{*}., ref=(\roman{*})]
\item for every $ 0<\beta < \frac 43$ there exists a constant $C_{\beta}>0$ such that 
\begin{equation}
\label{beta1}
\int_D |\nabla_yg(s,x,y)|^{\beta}{\rm d}y \le C_{\beta}s^{-\frac{3 \beta}{2}+1}\end{equation}
and 
\begin{equation}
\label{beta2}
\int_0^t \int_D| \nabla_yg(s,x,y)|^{\beta}{\rm d}y {\rm d}s \le C_{\beta}t^{-\frac{3 \beta}{2}+2}
\end{equation}
\item for every $ 0<\beta < 2$ there exists a constant $C_{\beta}>0$ such that
\begin{equation}
\label{beta_kernel_0}
\int_D|g(s,x,y)|^{\beta}{\rm d}y  \le C_{\beta} s^{1-\beta}
\end{equation}
and
\begin{equation}
\label{beta_kernel}
\int_0^t \int_D|g(s,x,y)|^{\beta}{\rm d}y {\rm d}s \le C_{\beta} t^{2-\beta}.
\end{equation}
\end{enumerate}
\end{theorem}
This result is proven in \ref{Appendix}.


\subsection{The Biot-Savart law}
\label{biot_savart_section}
Now we deal with the Biot-Savart law expressing the velocity vector field $v$ 
in terms of the vorticity scalar field $\xi$ (we mainly refer to \cite{Majda2002} 
and \cite{Marchioro2012}). 
We have $\xi = \nabla^{\perp} \cdot v$; by  taking the {\it curl} in both
sides of this relationship we get 
\begin{equation}
\begin{cases}
-\Delta v=\nabla^{\perp} \xi
\\
\nabla \cdot v=0
\\v \text{ periodic}
\end{cases}
\end{equation}
This allows to express the velocity in terms of the vorticity.
In terms of Fourier series, if
\begin{equation}
\label{vel_vor0}
\xi(x)=\frac{1}{2 \pi}\sum_{k \in \mathbb{Z}_0^2}\xi_ke^{ik \cdot x},
\end{equation}
then 
\begin{equation}
\label{vel_vor}
v(x) = -\frac{i}{2 \pi} \ \sum_{k \in \mathbb{Z}_0^2}  \xi_k \frac{k^\perp}{|k|^2}e^{i k \cdot x}.
\end{equation}
This shows that the velocity $v$ has one order more of
regularity with respect to the vorticity $\xi$: if $\xi
\in W^{b-1,p}$ then $ v \in H^b_p$.
In particular, the norms $\| v\|_{H^b_p}$ and $\|\xi\|_{W^{b-1,p}}$
are equivalent.

In general (see, e.g., \cite[Chapter 1]{Marchioro2012}), 
the Biot-Savart law expresses the velocity in term of the
vorticity as
\begin{equation}
\label{v}
v(x)= (k \ast \xi)(x)=\int_{D} k(x-y)\xi(y)\,{\rm d} y,
\end{equation}
where the Biot-Savart kernel is given by
\begin{equation}
\label{k} k= \nabla^{\perp}G= \left( -\frac{\partial G}{\partial x_2}, \frac{\partial G}{\partial x_1}\right)
\end{equation}
and $G$ is the Green function of the Laplacian on the torus  with mean zero.
Notice that from \eqref{v} it is evident that the relation between $ v$ and $\xi$ is non local in space.

We summarize the basic properties of the Biot-Savart kernel in the
following lemma (see \cite[Lemma 2.17]{Brzezniak2016}).
\begin{lemma}\label{stimaBS}
For every $1\le p<2$ the  map $ k$, defined above, is an $\left[L^p(D)\right]^2$ 
divergence-free (in the distributional sense) vector field.
\end{lemma}

\begin{remark}
\label{p_int_BS}
In principle, for every $p<2$, $ \int_{D} k(x-y)\, {\rm d} y$ is a constant that depends on $x$, but it can be easily majored by a constant which does not depend on $x$. This is straightforward using the estimate $|\nabla G(x)| \le C(|x|^{-1}+1)$ (see e.g. \cite[Chapter 1]{Marchioro2012} and \cite[Proposition B.1]{Brzezniak2016}) and recalling that \eqref{k} holds.
\end{remark}

Therefore we have some useful estimates. From \eqref{vel_vor}, using
the Sobolev embedding $H^1_p\subset L_\infty$ for $p>2$ and the
equivalence of the norms 
$\| v\|_{H^1_p}$ and $\|\xi\|_{L^p}$  we infer that for any
$p>2$ there exists a constant $C_p$ such that 
\begin{align}
\label{L_infty}
\|k\ast \xi\|_{L_{\infty}}=\| v\|_{L_{\infty}} \le   C_p\|\xi\|_{L^p}.
\end{align}
From \eqref{v} and Lemma \ref{stimaBS}, using Young's inequality when 
$p \ge 1$, $1 \le \alpha <2$, $\beta \ge 1$ with
 $\frac{1}{p}+1=\frac{1}{\alpha}+\frac{1}{\beta}$ 
we infer that
\begin{equation}
\label{_Bio_Sav_}
\|k\ast \xi\|_{L_p}  =  \| v\|_{L_{p}} \le \| k\|_{L_{\alpha}} \|\xi\|_{L^{\beta}}.
\end{equation}


\subsection{The random forcing term}
\label{hyp_res}
In this subsection we deal with the stochastic term that appears in \eqref{Walsh_0}.

Given $T>0$, let $(\Omega, \mathcal{F}, \mathbb{F}=\left\{ \mathcal{F}_t\right\}_{0\le t \le T}, \mathbb{P})$ be a given stochastic basis.
Let $Q: \dot L^2_{\sharp} \rightarrow \dot L^2_{\sharp}$ be a positive symmetric bounded linear operator.
We define $L^2_Q$ as the completition of the space of all square integrable, zero mean-value, periodic functions $\varphi:D \rightarrow \mathbb{R}$ with respect to the scalar product
\begin{equation*}
\langle \varphi, \psi \rangle_{L^2_Q}= \langle Q\varphi, \psi\rangle_{L^2}.
\end{equation*}
Set $\mathcal{H}_T=L^2(0,T;L^2_Q)$. 
This space is a real separable Hilbert space with respect to the scalar product 
\begin{equation}
\langle f,g\rangle_{\mathcal{H}_T}=\int_0^T\langle f(s), g(s)\rangle_{L^2_Q} \, 
{\rm d}s =\int_0^T\langle Qf(s), g(s)\rangle_{L^2} \, {\rm d}s.
\end{equation}
Let us consider the isonormal Gaussian process $W=\{W(h), h \in \mathcal{H}_T\}$
(see, e.g., \cite{Nualart2006}). 
The map $h \rightarrow W(h)$ provides a linear isometry from
 $\mathcal{H}_T$ onto $\mathcal{H}$, which is a closed subset of 
$L^2(\Omega, \mathcal{F}, \mathbb{P})$ whose elements are zero-mean Gaussian random variables.
The isometry reads as
\begin{equation*}
\mathbb{E}\left( W(h) \ W(g) \right) = \langle h, g \rangle_{\mathcal{H}_T}.
\end{equation*}

We understand the stochastic term appearing in equation \eqref{Walsh_0} in the following sense:
for $h \in \mathcal{H}_T$, we set
\begin{equation}
\label{W_h}
W(h)=\int_0^T \int_Dh(s,y) \, w({\rm d}y, {\rm d}s)
\end{equation}
namely, $\int_0^T \int_Dh(s,y) \, w({\rm d}y, {\rm d}s)$ 
is a zero-mean Gaussian 
random variable with covariance $\mathbb E[W(h)^2]=\|h\|^2_{\mathcal{H}_T}$.

We point out that the stochastic term introduced above by means of 
the linear isometry $W$ can be understood in the setting introduced 
by Da Prato-Zabczyk in \cite{DaPrato1992} as well as in the setting 
introduced by Walsh in \cite{Walsh1986}.
First,  we can write $W(h)$ as
\begin{equation} \label{Dap_Zab}
W(h)= \sum_{j} \int_0^T
 \langle h(s, \cdot),\tilde e_j \rangle_{L^2_Q} \, {\rm d} \beta_s(\tilde e_j)
\end{equation}
where $\{\tilde e_j\}_j$ is a complete orthonormal basis of $L^2_Q$ 
and  $\beta_s(\tilde e_j)=W(1_{[0,s]}\tilde e_j)$; hence
$\left\{\beta(\tilde e_j)\right\}_j$
 is 
a sequence of independent standard one-dimensional Brownian motions on 
$(\Omega, \mathcal{F}, \mathbb{P})$ adapted to $\{ \mathcal{F}_t\}_{t \in \left[0,T\right]}$.
By setting $M_t(A):=W(\pmb{1}_{\left[0,t \right]}\pmb{1}_A)$ for all 
$t \in \left[0,T\right]$ and $A \in \mathcal{B}_b(\mathbb{R}^2)$, 
we construct a martingale measure with covariance $Q$ and 
(see e.g. \cite{Dalang2011}) \eqref{Dap_Zab} coincides with the stochastic integral in the Walsh sense.
Moreover, the isonormal Gaussian process $W$ can be associated to a $Q$-Wiener process $\mathcal{W}_t$ on $\dot L^2_{\sharp}$ (as defined in \cite{DaPrato1992}) in the following way:%
\begin{equation}
\label{equality}
\langle \mathcal{W}_t, h\rangle_{L^2}=W(\pmb{1}_{\left[0,t\right]}h) \qquad \forall h \in \dot L^2_{\sharp},
\end{equation}
 and \eqref{W_h} 
coincides with the integral w.r.t. $\mathcal{W}$, in a sense made precise in \cite[Section 3.4]{Dalang2011}.
The stochastic convolution appearing in \eqref{Walsh_0} has now to be
understood in the described ways. Notice that, by construction, the random forcing term
is periodic and with zero mean in the space variable.
Since we are in a spatial domain of
dimension larger than one, it is not surprising (see, e.g.,
\cite{DaPrato1992}) that we cannot consider $Q$ to be the indentity,
but we need $Q$ to have some regularizing effect.
We choose to work with a covariance operator of the form 
\begin{equation}
\label{delta}
Q=(-\Delta)^{-b},
\end{equation}
for some $b>0$. This means that 
\[
Qe_k=|k|^{-2b}e_k \qquad \forall k \in \mathbb Z_0^2
\]
and a complete orthonormal basis of $L^2_Q$ is given by 
$\tilde e_k(x)=
\frac 1{\sqrt 2 \pi} |k|^b \cos(k\cdot x)$ 
and $\tilde e_{-k}(x)=\frac 1{\sqrt 2 \pi}|k|^b\sin(k\cdot x)$ 
for $k\in \mathbb Z^2_+$.
Notice that the choice of $Q$ as in \eqref{delta} is made only in order to simplify 
some computations but it does not  prevent to consider a more
general operator $Q$ which does not commute with the Laplacian operator or which has finite dimensional range. 
By Tr$Q$ we denote the trace of the operator $Q$. If $Q$ is as in \eqref{delta} then Tr$Q= \sum_{k \in \mathbb{Z}^2_0}|k|^{-2b}$.

Let us show that  when $b>0$ in \eqref{delta}  the stochastic
integral $\int_0^t \int_D g(t-s,x,y)\,  w({\rm d}y, {\rm d}s)
$ is well defined. This is equivalent to have
$g(t- \cdot, x, \cdot) \in \mathcal{H}_t$ for every $t>0$. Indeed, 
\begin{equation}\label{conv_stoc}
\begin{split}
\|g(t- \cdot, x, \cdot)\|^2_{\mathcal{H}_t}
&
=\int_0^t \|g(t-s,x,\cdot)  \|^2_{L^2_Q} \, {\rm d}s
=\int_0^t \langle Qg(t-s,x,\cdot), g(t-s,x,\cdot)\rangle_{L^2} \, {\rm d}s
\\
&
=  \int_0^t \|Q^{\frac12}g(t-s,x,\cdot)  \|^2_{L^2} \, {\rm d}s 
=\int_0^t \sum_{k \in \mathbb{Z}_0^2}|\langle e_k, Q^{\frac12} g(t-s,x,\cdot)\rangle_{L^2}|^2 \, {\rm d}s
\\
&=\sum_{k \in \mathbb{Z}_0^2}\int_0^t 
 \left| \langle Q^{\frac12}e_k, g(t-s,x,\cdot)\rangle_{L^2}\right|^2\, {\rm d}s
\\
&=\sum_{k \in \mathbb{Z}_0^2} |k|^{-2b} \int_0^t 
  \left|\langle e_k, g(t-s,x,\cdot)\rangle_{L^2}\right|^2\, {\rm d}s
 \\
&=\sum_{k \in \mathbb{Z}_0^2} |k|^{-2b} \int_0^t e^{-2|k|^2(t-s)}
 |e_k(x)|^2 \, {\rm d}s \qquad\text{ by } \eqref{kernel_fourier}
  \\
&=\frac{1}{(2\pi)^2}\sum_{k \in\mathbb{Z}_0^2}   |k|^{-2b}\frac{(1-e^{-2|k|^2t})}{2|k|^2} \qquad \text{since} \ |e_k(x)|=\frac{1}{2\pi}
\\
&\le \frac {1}{2(2\pi)^2}\sum_{k \in \mathbb{Z}_0^2}  |k|^{-2-2b}.
\end{split}
\end{equation}
The latter series is convergent if and only if $b>0$.
The hypothesis $b>0$ is sufficient for the space-time continuity of the 
stochastic convolution's trajectories as well (see \cite[Theorem 2.13]{DaPrato2004b}).
\begin{remark}
Notice that, since we work on the flat torus, we have good estimates on the norm 
of the normalized eigenfunctions $e_k$  of the Laplacian. 
Thanks to this fact we have rather weak assumptions on the covariance
operator of the noise, i.e. the exponent $b$ in \eqref{delta}. 
However, in a general domain of $\mathbb{R}^2$
with smooth boundary, the growth of normalized eigenfunctions is
more difficult to control. 
Useful estimates for this case are provided for instance in \cite{Grieser2002}.
\end{remark}


\section{Some preliminaries Lemmas}
\label{prliminary_lemmas}

In this Section we establish some  estimates
showing the regularizing effect of convolution with the gradient of the kernel $g$ or with $g$ itself, 
as they appear in the formulation \eqref{Walsh_0} \`a la Walsh of our problem.

Let $J$ be the linear operator defined as
\begin{equation}\label{defJfi}
(J \varphi)(t,x):=\int_0^t \int_D \nabla_y g(t-s,x,y) \cdot  \varphi (s,y) \, {\rm d}y\, {\rm d}s,
\end{equation}
for $t \in \left[0,T\right]$, $x \in D$.
We have that $J$ is well defined in some spaces as defined in the following lemma.
\begin{lemma}
\label{regularization}
i) 
Let $p\ge 1$, 
$\alpha \ge 1$, $1\le\beta < \frac 43$, $\gamma>\frac{2\beta}{2-\beta}$ 
such that $\frac{1}{\beta}=1 + \frac{1}{p}-\frac{1}{\alpha}$. \\
Then $J$ is a bounded linear operator from $L^{\gamma}(0,T; L_{\alpha})$ 
into $L^{\infty}(0,T;L^p)$. Moreover there exists a constant $C_\beta$  such that
\begin{equation}
\label{stima Lp0}
\|J( \varphi)(t, \cdot)\|_{L^p}\le C_\beta\int_0^t (t-s)^{\frac{1}{\beta}-\frac 32} \| \varphi(s, \cdot)\|_{L_{\alpha}} \, {\rm d}s,
\end{equation}
\begin{equation}
\label{stima Lp}
\|J(\varphi)(t, \cdot)\|_{L^p}
\le C_{\beta} t^{\frac{1}{\beta}-\frac 32+\frac{\gamma-1}{\gamma}}\left(\int_0^t\|\varphi(s, \cdot)\|^{\gamma}_{L_{\alpha}}\, {\rm d}s\right)^{\frac{1}{\gamma}}
\end{equation}
for all $t \in \left[0,T\right]$.

ii) Let $p>4$ and $\gamma >\frac{2p}{p-2}$.
Then the operator $J$ maps $L^{\gamma}(0,T;L_p)$ into $C(\left[0,T\right] \times D)$. Moreover there exists
a constant $C_{T,p}$ such that
\begin{equation}
\label{continuity}
\sup_{0\le t\le T}\sup_{x\in D}|(J  \varphi)(t,x)| \le C_{T,p}\left( \int_0^T\| \varphi(r, \cdot)\|^{\gamma}_{L_p}\, {\rm d}r\right)^{\frac{1}{\gamma}}.
\end{equation}
\end{lemma}
\begin{proof}
These results are inspired by  \cite[Lemma 3.1]{Gyongy1998}, but we need to perform all the computations since now we are in a two dimensional domain.

We first prove {\it i)}. 
Using the continuous version of  Minkowski's inequality (see e.g. \cite[Theorem 6.2.14]{Stroock1999}), then Young's inequality with $\frac{1}{\alpha}+ \frac{1}{\beta}=1+\frac{1}{p}$, and finally H\"older's inequality with $\gamma >\frac{2\beta}{2-\beta}$  we get
\begin{multline*}
\left\Vert \int_0^t \int_D  \nabla_yg(t-s, \cdot, y) \cdot  \varphi(s,y) \, {\rm d}y {\rm d}s\right \Vert_{L^p} 
\\
\le \int_0^t \left\Vert\int_D\nabla_yg(t-s, \cdot, y) \cdot  \varphi(s,y) \, {\rm d}y \right \Vert_{L^p}{\rm d}s
=\int_0^t \| \nabla_y g(t-s, 0,\cdot) * \varphi(s, \cdot)\|_{L^p}\, {\rm d} s
\\
 \le \int_0^t \|\nabla_y g(t-s,0, \cdot)\|_{L^{\beta}} \| \varphi (s, \cdot)\|_{L_{\alpha}}\, {\rm d}s
 \le C_\beta \int_0^t (t-s)^{\frac{1}{\beta}-\frac 32}\| \varphi(s, \cdot)\|_{L_{\alpha}}\, {\rm d}s \ \text{ by } \eqref{beta1}.
\end{multline*}
This proves \eqref{stima Lp0}. By H\"older's inequality we estimate the latter quantity by
\[
 C_\beta \left(\int_0^t (t-s)^{(\frac{1}{\beta}-\frac 32)\frac{\gamma}{\gamma-1}}{\rm d}s\right)^{\frac {\gamma-1}\gamma}
 \left(\int_0^t\| \varphi(s, \cdot)\|^\gamma_{L_{\alpha}}\, {\rm d}s\right)^{\frac 1\gamma}.
\]
Calculating the first time integral we obtain \eqref{stima Lp}.

As regards {\it ii)}, we use the factorization method (for more details see, e.g.,
 \cite[Section 2.2.1]{DaPrato2004b}), which is based on the 
equality
\begin{equation}
\label{factorization}
\frac{\pi}{\sin (\pi a)} = \int_s^t (t-r)^{a-1}(r-s)^{-a}\, {\rm d}r, \qquad a \in \left(0,1\right).
\end{equation}
We also use the  Chapman-Kolmogorov relation for $s<r<t$
\[
\int_D g(t-r,x,z)g(r-s,z,y)\, {\rm d}z=g(t-s,x,y)
\]
which, thanks to the symmetry of the kernel $g$ in the space variables,  gives
\begin{multline}\label{ugCK}
 \int_D\partial_{z_i}g(t-r,x,z) g(r-s,z,y)\, {\rm d}z
 =\int_D -\partial_{ x_i}g(t-r,x,z) g(r-s,z,y)\ {\rm d}z
 \\
 =-\partial_{x_i}\int_D g(t-r,x,z) g(r-s,z,y)\ {\rm d}z
 =-\partial_{x_i}g(t-s,x,y)=\partial_{y_i}g(t-s,x,y).    
 \end{multline}
Let us show that $J\varphi$, defined in \eqref{defJfi},  has an equivalent expression given by 
\begin{equation}
\label{J}
(J \varphi)(t,x)
=\frac{\sin (\pi a)}{\pi}\int_0^t (t-r)^{a-1}\left(\int_D \nabla_zg(t-r,x,z) \cdot  Y^a(r,z)\, {\rm d}z\right) {\rm d}r
\end{equation}
with
\[
 Y^a(r,z)=\int_0^r \int_D (r-s)^{-a}g(r-s,z,y) \varphi(s,y)\, {\rm d}y\, {\rm d}s.
\]
For this it is enough to check that 
\begin{multline}
\int_0^t \int_D \partial_{y_i} g(t-s,x,y)   \varphi_i (s,y) \, {\rm d}y\, {\rm d}s
\notag\\
=
\frac{\sin (\pi a)}{\pi}\int_0^t (t-r)^{a-1}\left(\int_D \partial_{z_i}g(t-r,x,z)   Y_i^a(r,z)\, {\rm d}z\right) {\rm d}r
\end{multline}
for $i=1,2$. Let us work on the r.h.s.; keeping in mind the definition of $Y^a_i$ and by means of  Fubini theorem
we infer that
\[\begin{split}
\int_0^t &(t-r)^{a-1}\left(\int_D \partial_{z_i}g(t-r,x,z)   Y_i^a(r,z)\, {\rm d}z\right) {\rm d}r
\\&=
\int_0^t (t-r)^{a-1}\left(\int_D \partial_{z_i}g(t-r,x,z)   \right.
\\
& \qquad \left.\Big[\int_0^r \int_D (r-s)^{-a}g(r-s,z,y) \varphi_i(s,y)\, {\rm d}y\, {\rm d}s\Big] {\rm d}z\right) {\rm d}r
\\&=
\int_0^t\left(\int_s^t (t-r)^{a-1}(r-s)^{-a}  \right.
\\
& \qquad \left.
\left[\int_D\Big[\int_D \partial_{z_i}g(t-r,x,z)  g(r-s,z,y) {\rm d}z \Big] \varphi_i(s,y)\, {\rm d}y\right]
   {\rm d}r\right) {\rm d}s
\\&=
\int_0^t\left(\int_s^t (t-r)^{a-1}(r-s)^{-a} \left[\int_D \partial_{y_i}g(t-s,x,y) \varphi_i(s,y)\, {\rm d}y\right]
   {\rm d}r\right) {\rm d}s \ \text{ by } \eqref{ugCK}
\\&=\frac\pi{\sin (\pi a)} \int_0^t\int_D \partial_{y_i}g(t-s,x,y) \varphi_i(s,y)\, {\rm d}y\ {\rm d}s
\qquad \text{ by } \eqref{factorization}.
\end{split}
\]
This proves \eqref{J}. Therefore, by H\"older's inequality
we get
\[\begin{split}
|(J \varphi)(t,x)|
&\le
\frac{\sin (\pi a)}{\pi}
\int_0^t (t-r)^{a-1} \| \nabla_zg(t-r,x,\cdot)\|_{L^{\frac p{p-1}}}\|Y^a(r,\cdot)\|_{L_p}{\rm d}r
\\&\le C_p
\frac{\sin (\pi a)}{\pi} \int_0^t (t-r)^{a-1-\frac 32 +\frac{p-1}p} \|Y^a(r,\cdot)\|_{L_p}{\rm d}r
\ \text{ by \eqref{beta1} if } p>4.
\end{split}
\]
Now we estimate $\| Y^a(r,\cdot)\|_{L_p}$; by means of Minkowsky's and Young's inequalities and using  \eqref{beta_kernel} we infer that
\[\begin{split}
\| Y^a(r,\cdot)\|_{L_p}&
=\left\| \int_0^r\ \int_D  (r-s)^{-a}g(r-s,\cdot,y) \varphi(s,y)\ {\rm d}y\  {\rm d}s\right\|_{L_p}
\\&\le
\int_0^r(r-s)^{-a}\left\| \int_D g(r-s,\cdot,y) \varphi(s,y)\,  {\rm d}y\right\|_{L_p}{\rm d}s
\\&
=\int_0^r(r-s)^{-a}\left\|  g(r-s,0,\cdot) \ast\varphi(s,\cdot)\right\|_{L_p}  {\rm d}s
\\&
\le \int_0^r(r-s)^{-a} \|g(r-s,0,\cdot)\|_{L_1}\|\varphi(s,\cdot)\|_{L_p}  {\rm d}s
\\&
\le C\int_0^r(r-s)^{-a}\|\varphi(s,\cdot)\|_{L_p}  {\rm d}s .
\end{split}\]
Collecting the above estimates, by means of Fubini theorem we obtain that
\begin{align*}
|(J \varphi)(t,x)|
&\le C\frac{\sin (\pi a)}{\pi}
\int_0^t (t-r)^{a-\frac32-\frac 1p}\left(\int_0^r(r-s)^{-a}\|\varphi(s,\cdot)\|_{L_p} {\rm d}s\right) {\rm d}r
\\
&= C\frac{\sin (\pi a)}{\pi}\int_0^t \| \varphi(s,\cdot)\|_{L^p}
\left(\int_s^t (t-r)^{a-\frac32-\frac 1p}(r-s)^{-a} {\rm d}r \right) {\rm d}s.
\end{align*}
With the change of variables $r=s+z(t-s)$ we can compute the inner integral as follows:
\begin{align*}
\int_s^t (t-r)^{a-\frac32-\frac 1p}(r-s)^{-a} {\rm d}r 
&=
(t-s)^{-\frac{p+2}{2p}}\int_0^1 (1-z)^{a-\frac32-\frac 1p} z^{-a}\ dz
\\
&=
(t-s)^{-\frac{p+2}{2p}}\int_0^1 (1-z)^{a-1-\frac {p+2}{2p}} z^{-a}\ dz.
\end{align*}
The latter integral is equal to 
the beta function $B\left(1-a,a-\frac{p+2}{2p}\right)$, which 
 is finite provided $\frac{p+2}{2p} < a <1$; therefore given $p>4$ we choose $a\in \left(\frac{p+2}{2p},1\right)$. Hence
\begin{align*}
|(J \varphi)(t,x)|  
&\le C_{p} \int_0^t (t-s)^{-\frac{p+2}{2p}} \|\varphi(s,\cdot)\|_{L_p}\ {\rm d}s
\\
&\le C_{p} 
\left(  \int_0^t (t-s)^{-\frac{p+2}{2p}\frac\gamma{\gamma-1}} {\rm d}s\right)^{\frac{\gamma-1}{\gamma}}\left(\int_0^t\|\varphi(s, \cdot)\|^{\gamma}_{L_p}\ {\rm d}s\right)^{\frac {1}{\gamma}}
\\&
\le C_{T,p} \left(\int_0^T\|\varphi(s, \cdot)\|^{\gamma}_{L_p}\ {\rm d}s \right)^{\frac {1}{\gamma}}
\end{align*}
for $\frac{p+2}{2p}\frac\gamma{\gamma-1}<1$, i.e. $\gamma>\frac {2p}{p-2}$.

The above estimate shows that
$J\varphi \in L^{\infty}(\left[0,T\right] \times D)$
 for every $\varphi \in L^{\gamma}(0,T;L_p)$. 
 It remains to prove that $J \varphi \in C(\left[0,T\right] \times D)$.
Let us notice that for step functions $ \varphi$, $J\varphi$ is a space-time continuous function; this follows from the well posedness of the integral $\int_0^t \int_D \nabla_yg(t-s,x,y) \, {\rm d}y\, {\rm d}s$ (let us recall that $\int_0^t \int_D |\nabla_yg(t-s,x,y)| \, {\rm d}y\, {\rm d}s < \infty$, see \eqref{beta2}). This kind of regularity can be then extended to every $ \varphi \in L^{\gamma}(0,T;L_p)$ by a standard approximation procedure.
\end{proof}

The second result concerns the stochastic integral in equation \eqref{Walsh_0}, i.e. the process
\begin{equation}
\label{reg_z}
z(t,x)= \int_0^t \int_D g(t-s,x,y)\, w({\rm d}y,{\rm d}s)
\end{equation}
solution of
\begin{equation}
\label{stoc}
\begin{cases}
\dfrac{\partial z}{\partial t}(t,x) -\Delta z(t,x)  = w({\rm d}x,{\rm d}t)
\\
z(0,x)=0.
\end{cases}
\end{equation}
We have
\begin{lemma}
\label{regol_conv}
Let $b>0$ in \eqref{delta} and $p>2$. Then 
\begin{equation}
\label{reg_z_Lp}
\mathbb{E} \left[\sup_{t \in \left[0,T\right]}\|z(t, \cdot)\|^p_{L^p}\right] < \infty.
\end{equation}
Moreover, $\mathbb{P}$-a.s. $z$ is  a continuous function on $\left[0,T\right] \times D$.
\end{lemma}
\begin{proof}
We use the factorization method. Given  $\alpha \in \left(0, \frac12\right)$ we can represent $z$ as
\begin{equation*}
z(t,x)=\frac{\sin (\pi \alpha)}{\pi} \int_0^t (t-\sigma)^{\alpha-1}
\left(\int_D g(t-\sigma,x,z)Z^\alpha(\sigma,z)\, {\rm d}z \right) {\rm d}\sigma
\end{equation*}
with
\begin{align*}
Z^\alpha(\sigma,z):=\int_0^{\sigma}(\sigma-s)^{-\alpha}g(\sigma-s,z,y)\, w({\rm d}y, {\rm d}s).
\end{align*}
From \eqref{conv_stoc} we know that $Z^\alpha(\sigma,z)$
is a zero-mean real gaussian random variable with covariance given by 
\begin{align*}
\mathbb{E}|Z^\alpha(\sigma,z)|^2&= \sum_{k \in \mathbb{Z}^2_0}|k|^{-2b}\int_0^{\sigma}e^{-2|k|^2(\sigma-s)}(\sigma-s)^{-2 \alpha}|e_k(z)|^2\, {\rm d}s 
\\
&\le \frac{2^{2\alpha-1}}{(2 \pi)^2}\Gamma(1-2\alpha) \sum_{k \in \mathbb{Z}^2_0} |k|^{2(2\alpha-1-b)}
\end{align*}
where the Gamma function is finite provided $\alpha \in \left(0, \frac 12\right)$. The
latter  series converges if and only if  $b>2\alpha$.
Therefore, from the gaussianity of $Z^\alpha$, there exists $C_p>0$ such that
\begin{equation*}
\mathbb{E}|Z^\alpha(\sigma,z)|^p \le C_p(\mathbb{E}|Z^\alpha(\sigma,z)|^2)^{\frac p2} < \infty
\end{equation*}
and we have 
\begin{equation}
\label{p}
 \int_0^T\mathbb{E}\|Z^\alpha(\sigma, \cdot)\|^p_{L^p}\, {\rm d}\sigma = \int_0^T\int_D\mathbb{E}|Z^\alpha(\sigma,z)|^p\, {\rm d}z\, {\rm d}\sigma <\infty.
 \end{equation}
 
From Minkowsky's, Young's and H\"older's inequalities we infer that
\[
\begin{split}
\|z(t, \cdot)\|_{L^p} 
&\le \frac{\sin (\pi \alpha)}{\pi} \int_0^t (t-\sigma)^{\alpha-1}
\left\|\int_D g(t-\sigma,\cdot,z)Z^\alpha(\sigma,z)\, {\rm d}z\right\|_{L^p} \, {\rm d}\sigma
\\
&\le \frac{\sin (\pi \alpha)}{\pi} \int_0^t (t-\sigma)^{\alpha-1}\|g(t-\sigma,0,\cdot)\|_{L^1}\|Z^\alpha(\sigma,\cdot)\|_{L^p}\, \, {\rm d}\sigma
\\
&\le \frac{\sin (\pi \alpha)}{\pi} \int_0^t (t-\sigma)^{\alpha-1}\|Z^\alpha(\sigma,\cdot)\|_{L^p} \, {\rm d}\sigma \qquad \text{by \eqref{beta_kernel_0}}
\\
&\le C_{T, \alpha} \left(\int_0^t \|Z^\alpha(\sigma,\cdot)\|^p_{L^p}\, \, {\rm d}\sigma\right)^{\frac1p}
\end{split}
\]
provided $p >\frac{1}{\alpha}$. 
Then
\begin{align*}
\mathbb{E}\left[ \sup_{t \in \left[0,T\right]}\|z(t, \cdot)\|^p_{L^p}\right]
 &\le 
 C_{T, \alpha}\mathbb{E}\left[ \sup_{t \in [0,T]}\int_0^t\|Z^\alpha(\sigma, \cdot)\|^p_{L^p}\, {\rm d}\sigma \right]
\\
&=C_{T,\alpha}\int_0^T\mathbb{E}\|Z^\alpha(\sigma, \cdot)\|^p_{L^p}\ {\rm d}\sigma<\infty
\end{align*}
for any $p>2$. 
 This proves \eqref{reg_z_Lp}.

 As regards the proof of the existence of a space-time continuous modification of $z$ it is similar to \cite[Theorem 2.13]{DaPrato2004b} and it follows from \cite[Lemma 2.12]{DaPrato2004b}. In fact, by the semigroup representation of the heat kernel we can write
 \begin{equation*}
 z(t,x)= \frac{\sin(\pi \alpha)}{\pi}
 \int_0^t (t-\sigma)^{\alpha-1}[S(t-\sigma)Z^\alpha(\sigma, \cdot)](x)\, {\rm d}\sigma, \qquad x \in D, t\in \left[0,T\right].
 \end{equation*}
 Since we are dealing with the heat kernel on a flat torus and we are working under the assumption $b>0$, we are in the framework given by \cite[Hypothesis 2.10]{DaPrato2004b}. Then it is sufficient to prove that $Z^\alpha \in L^{2m}(\left[0,T\right]\times D)$ for $m >\frac1\alpha$. This immediately follows from \eqref{p}.
\end{proof}


\section{Existence and uniqueness of the solution}
\label{exist_uniq_cont}

The main aim of this Section is to prove the existence and uniqueness of the solution to the SPDE \eqref{vort_0} as stated in Theorem \ref{existence_thm}. Since the derivative of $w$ is formal, we consider the equation in a weak sense, as in \cite{Walsh1986} for the stochastic heat equation.
In order to simplify the notation, recalling the relation between the vorticity scalar field $\xi$ and velocity vector field $ v$ given by the Biot-Savart law \eqref{v}, let us define the vector field 
$ q(\xi)=\xi\ (k\ast\xi)$, i.e.
\begin{equation}
\label{q}
[q(\xi)](x)=\xi(x) \int_D k(x-y)\xi(y)\, {\rm d} y.
\end{equation}
By means of H\"older's inequality, from \eqref{L_infty} if $p>2$ we know that 
\begin{equation}
\label{q_bis}
\|q(\xi)\|_{L_p} \le \|\xi\|_{L^p} \|k \ast \xi\|_{L_{\infty}} \le C_p\|\xi\|^2_{L^p}
\end{equation}
namely $q: L^p\to L_p$ for any $p>2$.
This allows to write system \eqref{vort_0}  in an equivalent form, where the velocity does not appear anymore.

Since $v=k\ast\xi$ is divergence free, for the nonlinear term in equation \eqref{vort_0} we have
\[
v\cdot \nabla \xi = \nabla\cdot (v \xi)= \nabla \cdot q(\xi).
\]
Therefore we give this definition of solution to system \eqref{vort_0}. 
This is a weak solution in the sense of PDE's, hence  involving test functions $\varphi$.
\begin{definition}
\label{weak_sol}
We say that an $\dot L^2_{\sharp}$-valued continuous $\mathcal{F}_t$-adapted stochastic process
 $\xi$ is a solution to \eqref{vort_0} if  it solves \eqref{vort_0} in the following sense:
for every $ t \in \left[0,T\right]$, $\varphi \in W^a$ with $a>2$ we have
\begin{multline}
\label{weak}
\int_D\xi(t,x)\varphi(x)\,{\rm d}x - \int_0^t \int_D \xi(s,x) \Delta \varphi(x)\,{\rm d}x\, {\rm d}s
- \int_0^t \int_D \ q(\xi(s, \cdot))(x) \cdot \nabla \varphi(x)\, {\rm d}x\, {\rm d}s
\\
= \int_D\xi_0(x)\varphi(x)\, {\rm d}x + \int_0^t \int_D \varphi(x)\,  w({\rm d}x,{\rm d}s)
\qquad
\end{multline}
$\mathbb{P}$-a.s.
\end{definition}
Notice that the non linear term is well defined since, using repeatedly H\"older's inequality and the Sobolev embedding, we obtain
\[\begin{split}
\left|\int_D q(\xi(s, \cdot))(x) \cdot \nabla \varphi(x)\, {\rm d}x \right| 
&\le  \|\nabla \varphi\|_{L^{\infty}}\|q(\xi(s, \cdot))\|_{L_1} 
\\&
\le   C \|\nabla \varphi\|_{W^s} \| k \ast \xi(s, \cdot)\|_{L_2} \|\xi(s, \cdot)\|_{L^2} 
    \quad \text{ if } s>1
\\
&\le C \|\varphi\|_{W^{s+1}} \|\xi(s, \cdot)\|^2_{L^2} \ \text{ by } \eqref{_Bio_Sav_} \ (\alpha=1, \ \beta=p=2).
\end{split}
\]

Following the idea of  \cite{Walsh1986} for the heat equation or of 
 \cite{Gyongy1998} for the Burgers equation one obtains  that this is equivalent to ask 
that  for any $(t,x)\in [0,T]\times D$ 
\begin{multline}
\label{Walsh}
\xi(t,x)=\int_D g(t,x,y)\xi_0(y)\, {\rm d}y 
+ \int_0^t \int_D \nabla_yg(t-s, x,y) \cdot  q(\xi(s,\cdot))(y) \, {\rm d}y \,{\rm d}s
\\
+
\int_0^t \int_D g(t-s,x,y)\,  w({\rm d}y, {\rm d}s)
\end{multline}
$\mathbb{P}$-a.s. 

The non linear term $q(\xi)$ that appears in \eqref{Walsh} is non Lipschitz.
Therefore, we use a localization argument
to prove the existence and uniqueness of the solution. By means of a fixed point argument we prove at first the existence and uniqueness result for a local solution; then the global result follows from suitable estimates on the process $\xi$.

So, we first solve the problem when the nonlinearity is truncated to be globally Lipschitz.

\subsection{The case of truncated nonlinearity}
Let $N\ge 1$ and denote by $\Theta_N: \left[0,+\infty\right) \rightarrow [0,1]$ a $C^1$ function such that 
$|\Theta_N^\prime(s)|\le 2$ for any $s\ge 0$ and 
\begin{equation}
\label{truncation_factor}
\Theta_N(s)=
\begin{cases}
1 \qquad \text{if} \ 0\le s < N
\\
0 \qquad  \text{if} \ s \ge N+1
\end{cases}
\end{equation}
Given $\xi \in L^p$, for $p>2$, we define 
\begin{equation}
\label{q_N}
q_N(\xi)= q(\xi)  \Theta_N(\|\xi\|_{L^p}),
\end{equation}
\begin{equation}
\label{q_tilde}
\tilde q_N(\xi)= q(\xi)  \Theta_N^{\prime}(\|\xi\|_{L^p}).
\end{equation}
By \eqref{q_bis} we know that $q_N, \tilde q_N: L^p\to L_p$ for any $p>2$. In addition we have
\begin{lemma}
\label{Lipschitz}
Fix $N\ge 1$ and $p>2$. Then 
there exist positive constants $C_p$ and $L_{N,p}$ such that
\begin{equation}
\label{stima_q}
\| q_N(\xi)\|_{L_p}  \le C_p(N+1)^2 \qquad \forall \xi \in L^p,
\end{equation}
\begin{equation}
\label{stima_q_2}
\| \tilde q_N(\xi)\|_{L_p}  \le C_p(N+1)^2 \qquad \forall \xi \in L^p
\end{equation}
and 
\begin{equation}\label{qlip}
\| q_N(\xi)-  q_N(\eta)\|_{L_p} \le L_{N,p}\|\xi-\eta\|_{L^p} \qquad \forall \xi, \eta \in L^p.
\end{equation}
\end{lemma}
\begin{proof}
The global bounds comes from \eqref{q_bis}:
\[
\| q_N(\xi)\|_{L_p} 
\le  C_p\|\xi\|^2_{L^p}\Theta_N(\|\xi\|_{L^p}) 
\le C_p(N+1)^2,
\]

\[
\| \tilde q_N(\xi)\|_{L_p} 
\le  C_p\|\xi\|^2_{L^p}|\Theta^{\prime}_N(\|\xi\|_{L^p}) |
\le C_p(N+1)^2.
\]

Let us now show that $ q_N$ is a Lipschitz continuous function. The idea is to use the mean value theorem: we show that $q_N$ is G\^ateaux differentiable in any point of $L^p$ and 
its derivative is bounded. The result will follow  by
\begin{equation}\label{meanv}
\| q_N(\xi)-  q_N(\eta)\|_{L_p} 
\le\sup_{t \in \left[0,1\right]}\|D q_N(t \xi+(1-t)\eta)\|_{\mathcal{L}(L^p;L_p)}\|\xi-\eta\|_{L^p}
\end{equation}
where
$D q_N(\xi):h \rightarrow D_h  q_N(\xi)$ is a linear and bounded operator from $L^p$ into $L_p$ defined as
\begin{equation}
\label{defDh}
D_h q_N(\xi):= \lim_{\varepsilon \rightarrow 0}\frac{ q_N(\xi + \varepsilon h)-  q_N(\xi)}{\varepsilon}
\end{equation}
and 
\begin{equation*}
\|D q_N(\xi)\|_{\mathcal{L}(L^p;L_p)} = \sup_{\|h\|_{L^p}\le 1}\|D_h  q_N(\xi)\|_{L_p}.
\end{equation*}
By \eqref{defDh} we have
\begin{equation*}
D_h q_N(\xi)= q(\xi) \ D_h \Theta_N(\|\xi\|_{L^p}) + h \ ( k \ast \xi)
\Theta_N(\|\xi\|_{L^p}) +\xi \ ( k \ast h) \Theta_N(\|\xi\|_{L^p}) .
\end{equation*}
Since  $D_h(\|\xi\|_{L^p})= \|\xi\|^{1-p}_{L^p}\langle\xi |\xi|^{p-2},h\rangle$ we get
\[
D_h \Theta_N(\|\xi\|_{L^p}) = \Theta'_N(\|\xi\|_{L^p})  \|\xi\|^{1-p}_{L^p}\langle\xi |\xi|^{p-2},h\rangle.
\]
Therefore, bearing in mind \eqref{L_infty} 
and \eqref{stima_q_2} we infer that
\[
\begin{split}
\| D_h & q_N(\xi)\|_{L_p} 
\\&\le |\Theta^\prime_N(\|\xi\|_{L^p})| \|\xi\|_{L^p}^{1-p} |\langle \xi |\xi|^{p-2},h\rangle| \  \|q(\xi)\|_{L_p}
\\
& \qquad+| \Theta_N(\|\xi\|_{L^p}) | \|h( k \ast \xi)\|_{L_p}
+ |\Theta_N(\|\xi\|_{L^p})| \|\xi( k \ast h)\|_{L_p}
\\&\le
 |\Theta^\prime_N(\|\xi\|_{L^p})| \|h\|_{L^p} \|q(\xi)\|_{L_p}
 \\
&\qquad+ |\Theta_N(\|\xi\|_{L^p}) | \|h \|_{L^p} \|k \ast \xi\|_{L_\infty}
 +| \Theta_N(\|\xi\|_{L^p})|  \|\xi \|_{L^p} \|k \ast h\|_{L_\infty}
 \\&\le
 \|h\|_{L^p}\|\tilde q_N(\xi)\|_{L_p}
+ 2C_p |\Theta_N(\|\xi\|_{L^p}) | \|h \|_{L^p} \|\xi\|_{L^p}
\\&\le
 C_p (N+1)^2\|h \|_{L^p}+2 C_p(N+1)\|h \|_{L^p} .
\end{split}
\]
Hence we get
\[
\sup_\xi \|D q_N(\xi)\|_{\mathcal{L}(L^p;L_p)}\le  C_p (N+1)^2+2 C_p(N+1) .
\]
Thanks to \eqref{meanv} this proves \eqref{qlip}.
\end{proof}

We aim at proving the existence and uniqueness of the solution to the smoothed version of system \eqref{vort_0} that is
\[
\begin{cases}
\displaystyle  \frac{\partial\xi_N}{\partial t}(t,x)- \Delta \xi_N(t,x)+v_N(t,x) \cdot \nabla \xi_N(t,x) \Theta_N(\|\xi_N(t,\cdot)\|_{L^p})
       = w({\rm d}x,{\rm d}t)
       \\ 
\nabla \cdot  v_N(t,x)=0 
\\ 
\xi_N(t,x) = \nabla^{\perp} \cdot v_N(t,x)
\\
\displaystyle \xi_N(0,x)=\xi_0(x) 
\end{cases}
\]
Thanks to \eqref{q} and \eqref{q_N} this can be written  in the  Walsh formulation as
\begin{multline}
\label{Walsh_truncated}
\xi_N(t,x)=\int_D g(t,x,y)\xi_0(y)\, {\rm d}y 
+ \int_0^t \int_D \nabla_yg(t-s, x,y) \cdot q_N(\xi_N(s,\cdot))(y) \, {\rm d}y \,{\rm d}s
\\
+
\int_0^t \int_D g(t-s,x,y)\,  w({\rm d}y, {\rm d}s).\qquad
\end{multline}
We have the following result.
\begin{proposition}
\label{result_truncated}
Let $N\ge 1$, $b>0$ in \eqref{delta} and 
$p>2$. If $\xi_0 \in L^p$, then  there exists a unique solution 
$\xi_N$ to equation \eqref{Walsh_truncated} which is  an $\mathcal{F}_t$-adapted process whose paths belong to
$C([0,T];L^p)$, $\mathbb P$-a.s.
\end{proposition}
\begin{proof}
Since we are dealing with an additive noise, i.e. a noise which is independent of the unknown process $\xi$, we can work pathwise.
 In order to prove the existence and uniqueness result we appeal to the contraction principle. 
Let $\mathcal{B}$ denote the space of all $L^p$-valued $\mathcal{F}_t$-adapted stochastic processes $\eta(t, \cdot)$, $t \in \left[0,T \right]$ such that the norm 
\begin{equation*}
\|\eta\|_{\mathcal{B}}:= \sup_{t \in \left[0,T\right]}\|\eta(t, \cdot)\|_{L^p} 
\end{equation*}
is finite $\mathbb{P}$-a.s.

Define the operator $\mathcal{M}$ on $\mathcal{B}$ by
\begin{equation*}
\mathcal{M}(\xi_N)(t,x):= M_0(t,x)+(J q_N(\xi_N))(t,x)+z(t,x),
\end{equation*}
where
\begin{equation*}
M_0(t,x):= \int_Dg(t,x,y)\xi_0(y)\, {\rm d}y,
\end{equation*}
and the other two terms are given respectively by \eqref{defJfi} and \eqref{reg_z}. More precisely,
\begin{equation*}
(J q_N(\xi_N))(t,x)=\int_0^t \int_D\nabla_yg(t-s,x,y) \cdot q_N(\xi_N(s, \cdot))(y)\, {\rm d}y \,{\rm d}s.
\end{equation*}
Then
\begin{equation*}
\|\mathcal{M}(\xi_N)(t, \cdot)\|_{L^p}^p \le C_p\left(\|M_0(t, \cdot)\|^p_{L^p}+\|(Jq_N(\xi_N))(t, \cdot)\|^p_{L^p}+\|z(t, \cdot)\|^p_{L^p}\right).
\end{equation*}
Using Young's inequality and \eqref{beta_kernel_0}, we infer that
\begin{align*}
\|M_0\|_{\mathcal{B}}= \sup_{t \in \left[0,T\right]}  \|M_0(t, \cdot)\|_{L^p}
\le  \sup_{t \in \left[0,T\right]} \left(\|g(t,0, \cdot)\|_{L^1}\ \|\xi_0\|_{L^p} \right) < \infty.
\end{align*}
By estimates \eqref{stima Lp0} (with $\beta=1$, $\alpha=p$) and \eqref{stima_q} 
we get
\begin{equation*}
\|(Jq_N(\xi_N))(t,\cdot)\|_{L^p} 
\le  \int_0^t (t-s)^{-\frac12}\| q_N(\xi_N(s, \cdot))\|_{L_p} \, {\rm d}s
\le C_p  (N+1)^{2} t^{\frac 12}
\end{equation*}
and so $\|Jq_N(\xi_N)\|_{\mathcal{B}}<\infty$.
Finally, $\|z\|_{\mathcal{B}}< \infty$ $\mathbb{P}$-a.s. by Lemma \ref{regol_conv}.
Thus $\mathcal{M}$ is an operator mapping the Banach space $\mathcal{B}$ into itself. 
It remains to prove that $\mathcal{M}$ is a contraction. 
From \eqref{stima Lp0} with $\alpha=p$, $\beta=1$ and the Lipschitz result of Lemma \ref{Lipschitz}, 
we infer that
\[
\begin{split}
\|\mathcal{M}(\xi_N^1)(t, \cdot)-\mathcal{M}&(\xi_N^2)(t, \cdot)\|_{L^p} 
\le
C \int_0^t (t-s)^{-\frac12}\|q_N(\xi_N^1(s,\cdot))-q_N(\xi_N^2(s,\cdot))\|_{L_p}ds
\\&
  \le 
  C L_{N,p} \int_0^t (t-s)^{-\frac12}\|\xi_N^1(s, \cdot)-\xi_N^2(s, \cdot)\|_{L^p}\, {\rm d}s
\\&
\le
C L_{N,p} \left(\sup_{t \in \left[0,T\right]}\|\xi_N^1(t, \cdot)-\xi_N^2(t, \cdot)\|_{L^p} \right)
\int_0^t (t-s)^{-\frac12} {\rm d}s
\\&\le 
C_{N,p} T^{\frac 12} \|\xi_N^1-\xi_N^2\|_{\mathcal{B}}
\end{split}
\]
for every $t \in \left[0,T \right]$. If $T$ satisfies $C_{N,p}T^{\frac 12}< 1$, 
then $\mathcal{M}$ is a contraction on $\mathcal{B}$. Hence the operator $\mathcal{M}$ admits a unique fixed point in the set $\{\xi \in \mathcal{B}: \xi(0,\cdot)=\xi_0\}$. 
Otherwise we choose $\tilde t>0$ such that  $C_{N,p}\tilde t^{\frac 12}< 1$ and we conclude the existence of a unique solution on the time interval $[0,\tilde t]$.
Since $C_{N,p}$ does not depend on $\xi_0$, by a standard argument we construct a 
unique solution $\xi$ to the SPDE \eqref{Walsh_truncated} by concatenation on 
every interval of lenght $\tilde t$ until we recover the time interval $\left[0,T\right]$.
\end{proof}

In the following subsection we shall see that Proposition \ref{result_truncated} provides uniqueness and local existence for the solution in Theorem \ref{existence_thm}. To gain the global existence we need a uniform estimate as proved in the following lemma, inspired by \cite{Gyongy1998} and \cite{GyongyNualart1999}.

Let $z$ be the process defined in \eqref{reg_z} and $\xi_N$ be the solution to equation \eqref{Walsh_truncated}. Let us set $\beta_N=\xi_N-z$. Since the noise is independent on the unknown, $\beta_N$ satisfies the equation
\begin{equation}
\label{eq_truncated_2}
\frac{\partial}{\partial t}\beta_N 
=\Delta \beta_N - \nabla \cdot  q_N(\beta_N+z)
\end{equation}
which can be written \`a la Walsh as 
\begin{multline}
\label{Walsh_truncated_2}
\beta_N(t,x)= \int_Dg(t,x,y)\xi_0(y)\, {\rm d}y 
\\
+ \int_0^t \int_D\nabla_yg(t-s,x,y) \cdot  q_N(\beta_N(s, \cdot)+ z(s, \cdot))(y)\, {\rm d}y \, {\rm d}s.
\end{multline}
The following result provides a uniform estimate for $\beta_N$. 
Notice that we shall work pathwise since the noise is additive.

\begin{lemma}
\label{stima_limi}
Let $b>0$ in \eqref{delta} and $p>2$. 
If $\xi_0 \in L^p$ then 
\begin{align*}
\sup_{N \ge 1} \sup_{t \in \left[0,T\right]}\|\beta_N(t, \cdot)\|^p_{L^p} \le \left[\|\xi_0\|^p_{L^p}+C_1(z) \right]e^{C_2(z)}
\end{align*}
where $C_1(z)$ and $C_2(z)$ are given by
\begin{align*}
C_1(z)=C_pT \sup_{t \in \left[0, T\right]}\|z(t, \cdot)\|^{2p}_{L^p}
\end{align*}
and
\begin{align*}
C_2(z)=C_pT\left(1+ \sup_{t \in \left[0,T\right]}\|z(t, \cdot)\|^2_{L^p}\right)
\end{align*}
for some positive constant $C_p$.
\end{lemma}
\begin{proof}
As done before, we can show that a solution to \eqref{Walsh_truncated_2} is a weak solution to the PDE \eqref{eq_truncated_2} with initial condition $\beta_N(0,x)=\xi_0(x)$. 

We consider the time evolution of the  $L^p$-norm of $\beta_N(t,\cdot)$. 
Let us recall that when $b>0$ in \eqref{delta}, $z$ admits a modification with $\mathbb{P}$-a.s. space-time continuous trajectories; 
moreover from Proposition \ref{result_truncated} we know that $\xi_N \in C([0,T];L^p)$ $\mathbb{P}$-a.s. Hence, for sure, the solution 
$\beta_N \in C([0,T];L^p)$ $\mathbb{P}$-a.s. for every $N\ge 1$. Actually, since the noise term has disappeared, $\beta_N$ is more regular than $\xi_N$ and $z$. 
Indeed, $\frac{\partial \beta_N}{\partial t}-\Delta \beta_N=-\nabla \cdot q_N(\xi_N)$ where $q_N(\xi_N)$ belongs at least 
to $L^2(0,T;L_2)$ thanks to \eqref{stima_q}. Hence, according to a classical regularity result for parabolic equations 
(see e.g. \cite[Chapter 4.4, Theorem 4.1]{Lions1972})
we have that  $\beta_N \in L^2(0,T;W^1)$; hence $\nabla \beta_N$ exists.
We use this fact in the following computations. Only at the end we will obtain an estimate involving $\nabla \beta_N$ which shows 
its regularity.
This is a short way to prove our result. Otherwise one has to use Galerkin approximations and then pass to the limit.
 
From \eqref{eq_truncated_2} we infer that
\[
\begin{split}
\label{normaLp}
 \frac{d}{dt} \|\beta_N(t,\cdot)\|^p_{L^p}
 &=p \int_D |\beta_N(t,x)|^{p-2}\beta_N(t,x) \frac{\partial }{\partial t} \beta_N(t,x) \ dx
 \\
 &=
 p\int_D |\beta_N(t,x)|^{p-2}\beta_N(t,x) \Delta \beta_N(t,x) \ dx 
 \\&\qquad
 - p
 \int_D |\beta_N(t,x)|^{p-2}\beta_N(t,x) \nabla \cdot  q_N(\beta_N(t,\cdot)+z(t,\cdot))(x) \ dx
\end{split}
\]
Integrating by parts the two latter integrals we obtain (writing for short $\beta_N(t)$ instead of $\beta_N(t,\cdot)$)
\begin{multline*}
\frac{d}{dt} \|\beta_N(t)\|^p_{L^p}
+p(p-1)\||\beta_N(t)|^{\frac{p-2}{2}}\nabla \beta_N(t)\|^2_{L^2}
\\
=p(p-1) \langle |\beta_N(t)|^{p-2},\nabla \beta_N(t) \cdot q_N(\beta_N(t)+z(t))\rangle.
\end{multline*}
We need to work on the latter term. 
Let us  write the quadratic term $q_N(\beta_N(t)+z(t))$ in the form $\Theta_N(\|\beta_N(t)+z(t)\|_{L^p}) k\ast (\beta_N(t)+z(t)) \ (\beta_N(t)+z(t))
= \Theta_N(\|\beta_N(t)+z(t)\|_{L^p}) k\ast (\beta_N(t)+z(t)) \ \beta_N(t)+
\Theta_N(\|\beta_N(t)+z(t)\|_{L^p}) k\ast (\beta_N(t)+z(t)) \ z(t)$; then using the basic property
$\langle  |\beta_N(t)|^{p-2}\beta_N(t),\nabla \beta_N(t)\cdot v(t) \rangle =0$ (where $v$ is a divergence free velocity field; this is 
obtained again by integration by parts, see  for instance \cite[Lemma 2.2]{FerrarioBessaih2014})
we obtain 
\begin{multline}
\label{ast}
\frac{d}{dt} \|\beta_N(t)\|^p_{L^p} 
+p(p-1)\||\beta_N(t)|^{\frac{p-2}{2}}\nabla \beta_N(t)\|^2_{L^2} 
\\=
 p(p-1)\langle \Theta_N(\|\beta_N(t)+z(t)\|_{L^p}) |\beta_N(t)|^{p-2}z(t),  \nabla \beta_N(t)\cdot[ k\ast (\beta_N(t)+z(t))]\rangle.
\end{multline}
Let us estimate the r.h.s., using H\"older's and Young's inequalities.
\begin{align*}
&\left| \langle\Theta_N(\|\beta_N(t)+z(t)\|_{L^p})  |\beta_N(t)|^{p-2}z(t),  \nabla \beta_N(t)\cdot[ k\ast (\beta_N(t)+z(t))]\rangle \right|
\\
&\le \left| \Theta_N(\|\beta_N(t)+z(t)\|_{L^p}) \right| \||\beta_N(t)|^{\frac{p-2}{2}} \nabla \beta_N(t)\|_{L^2}
\\
&\qquad \qquad \||\beta_N(t)|^{\frac{p-2}{2}} z(t)\|_{L^2} \|k \ast \left(\beta_N(t)+z(t)\right)\|_{L_{\infty}}
\\
&\le C_p \||\beta_N(t)|^{\frac{p-2}{2}} \nabla \beta_N(t)\|_{L^2} \||\beta_N(t)|^{\frac{p-2}{2}} z(t)\|_{L^2} \|\beta_N(t)+z(t)\|_{L^p} \qquad \text{by \eqref{L_infty}}
\\
&\le C_p\| |\beta_N(t)|^{\frac{p-2}{2}} \nabla \beta_N(t)\|_{L^2} \|\beta_N(t)\|^{\frac{p-2}{2}}_{L^p} \|z(t)\|_{L^p} \left( \|\beta_N(t)\|_{L^p}+ \|z(t)\|_{L^p}\right)
\\
&\le \frac12 \||\beta_N(t)|^{\frac{p-2}{2}} \nabla \beta_N(t)\|^2_{L^2}+ C_p \|\beta_N(t)\|^p_{L^p}\|z(t)\|^2_{L^p}
\\
&\qquad \qquad + C_p\|\beta_N(t)\|^p_{L^p}+ C_p \|z(t)\|^{2p}_{L^p}.
\end{align*}

Coming back to equation \eqref{ast}, we have obtained that
\begin{multline}
\label{ast_2}
\frac{d}{dt} \|\beta_N(t)\|^p_{L^p} 
+\frac{p(p-1)}{2}\||\beta_N(t)|^{\frac{p-2}{2}}\nabla \beta_N(t)\|^2_{L^2} 
\\
\le C_p\left(1+\|z(t)\|^2_{L^p} \right) \|\beta_N(t)\|^p_{L^p} + C_p \|z(t)\|^{2p}_{L^p}.
\end{multline}
Using Gronwall lemma on the inequality 
\begin{equation*}
\frac{d}{dt} \|\beta_N(t)\|^p_{L^p} 
\le C_p\left(1+\|z(t)\|^2_{L^p} \right) \|\beta_N(t)\|^p_{L^p} + C_p \|z(t)\|^{2p}_{L^p}
\end{equation*}
we obtain
\begin{align*}
\|\beta_N(t)\|^p_{L^p} 
&\le \|\xi_0\|^p_{L^p} e^{C_p \int_0^t \left( 1+ \|z(s)\|^2_{L^p}\right)\, {\rm d}s} + C_p\int_0^t e^{C_p\int_r^t \left( 1+\|z(s)\|^2_{L^p} \right)\, {\rm d}s}\|z(r)\|^{2p}_{L^p}\, {\rm d}r
\\
&\le e^{C_pT\left(1+ \sup_{0 \le s\le T}\|z(s)\|^2_{L^p}\right)} \left( \|\xi_0\|^p_{L^p}+C_pT \sup_{0 \le r\le T}\|z(r)\|^{2p}_{L^p} \right).
\end{align*}
Integrating in time \eqref{ast_2}, we obtain that $|\beta_N|^{\frac{p-2}{2}}\nabla \beta_N \in L^2(0,T;L^2)$ which is the regularity we expected.

\end{proof}

\subsection{Existence and uniqueness of the solution to \eqref{Walsh}}
We go back to the original equation \eqref{vort_0} in the form given by \eqref{Walsh} and prove the existence and uniqueness result stated in Theorem \ref{existence_thm}.
\begin{proof}[Proof of Theorem \ref{existence_thm}.]

Pathwise uniqueness is provided in a classical way by a stopping time argument. More precisely, suppose that $\xi^1$ and $\xi^2$ are two solutions to equation \eqref{vort_0}. Both satisfy \eqref{Walsh} thanks to the equivalence between the formulations \eqref{weak} and \eqref{Walsh}.
Let $p>2$; let us define the stopping times
\begin{equation*}
\tau^i_N:= \inf\{t \ge 0: \|\xi^i(t, \cdot)\|_{L^p}\ge N\} \wedge T, \qquad i=1,2,
\end{equation*}
for every $N\ge 1$ and let us set $\tau_N^*:= \tau_N^1 \wedge \tau_N^2$. Setting $\xi^i_N(t)=\xi^i(t \wedge \tau^*_N)$ for $i=1,2$, for all $t \in \left[0,T\right]$ we have that the processes $\xi^1_N$ and $\xi^2_N$ satisfy \eqref{Walsh_truncated}; hence, by the uniqueness result given by Proposition \ref{result_truncated}, $\xi^1_N=\xi^2_N$ $\mathbb{P}$-a.s. for all $t \in \left[0,T\right]$, that is $\xi^1=\xi^2$ on $\left[0, \tau_N^*\right)$ $\mathbb{P}$-a.s. Since $\tau_N^*$ converges $\mathbb{P}$-a.s. to $T$, as $N$ tends to infinity, we deduce $\xi^1=\xi^2$ $\mathbb{P}$-a.s for every $t \in \left[0,T\right]$.

Let us now prove the existence of the solution in $\left[0,T\right]$.
Let $p>2$; let us define the stopping time 
\begin{equation}
\label{sigma_N}
\sigma_N:= \inf\{t \ge 0: \|\xi_N(t, \cdot)\|_{L^p} \ge N\}\wedge T,
\end{equation}
for every $N\ge1$. In Proposition \ref{result_truncated} we have shown the global existence and uniqueness of the solution $\xi_N$ to the truncated problem \eqref{Walsh_truncated}. By uniqueness it follows that, given $M>N$ we have $\xi_N(t, \cdot)=\xi_M(t, \cdot)$ for $t\le \sigma_N$; so we can define a process $\xi$ by $\xi(t, \cdot)=\xi_N(t, \cdot)$ for $t \in \left[0, \sigma_N\right]$. Set $\sigma_{\infty}:= \sup_{N\ge1}\sigma_N$, then Proposition \ref{result_truncated} tells us that we have constructed a solution to \eqref{Walsh_truncated} in the random interval $\left[0, \sigma_{\infty}\right)$, and it is unique. To conclude, we just need to prove that 
\begin{equation}
\label{P}
\sigma_{\infty}=T \qquad \mathbb{P}-a.s. 
\end{equation}
that is equivalent to verify that 
\begin{equation*}
\lim_{N\rightarrow \infty}\mathbb{P}(\sigma_N<T)=0.
\end{equation*}
By Lemma \ref{stima_limi} we have that, for every $N \ge 1$,
\begin{equation*}
\sup_{t \in \left[0,T\right]}\log\|\beta_N(t, \cdot)\|_{L^p} \le \frac1p \log(\|\xi_0\|^p_{L^p}+  C_1(z))+ \frac{C_2(z)}{p}
\end{equation*}
and $\mathbb{E}\left[ C_1(z)\right]$, $\mathbb{E}\left[C_2(z)\right]$ are finite, according to Lemma \ref{regol_conv}.
Hence, for all $N\ge 1$,
\begin{equation*}
\mathbb{E}\left[ \sup_{t \in \left[0,T\right]}\log \|\beta_N(t, \cdot)\|_{L^p}\right]\le C_{p,T}\left( 1+ \log\|\xi_0\|^p_{L^p}\right) < \infty
\end{equation*}
by means of Jensen's inequality.
By Chebychev's inequality, it follows that
\begin{align*}
\mathbb{P}(\sigma_N<T)
&=\mathbb{P}\left(\sup_{t \in \left[0,T\right]}\|\xi_N(t, \cdot)\|_{L^p} >N\right) 
\\
&\le \mathbb{P}\left(\sup_{t \in \left[0,T\right]}\|\beta_N(t, \cdot)\|_{L^p} >\frac N2\right)+
\mathbb{P}\left(\sup_{t \in \left[0,T\right]}\|z(t, \cdot)\|_{L^p} >\frac N2\right)
\\
&\le \frac{1}{\log \left(\frac N2\right)}\mathbb{E} \left[\sup_{t \in \left[0,T\right]}\log\|\beta_N(t, \cdot)\|_{L^p}\right]+ \frac 2N \mathbb{E} \left[ \sup_{t \in \left[0,T\right]}\|z(t, \cdot)\|_{L^p}\right]
\\
&\le \frac{C_{p,T}\left(1+ \log \|\xi_0\|^p_{L^p}\right)}{\log N} + \frac{\hat C_{p,T}}{N},
\end{align*}
for some constant $C_{p,T}$ and $\hat C_{p,T}$. Then we obtain that $\lim_{N\rightarrow \infty}\mathbb{P}(\sigma_N<T)=0$.

Now, we assume $\xi_0$ to be continuous; then the solution $\xi$ given by \eqref{Walsh} is the sum of three terms. The first one, $\int_Dg(t,x,y)\xi_0(y)\, {\rm d}y $ is continuous by the properties of $g$ (see Theorem \ref{beta_gradient}\textit{(ii)}). As regards the second one, since $\xi_0\in C(D)$, then $\xi_0 \in L^{\tilde p}$ for any $\tilde p$. Choosing a value of $\tilde p >4$, we find that $q(\xi)\in C(\left[0,T\right];L_{\tilde p})$ and Lemma \ref{regularization}\textit{(ii)} provides that $Jq(\xi) \in C(\left[0,T\right] \times D)$.
Finally the third term is continuous thanks to Lemma \ref{regol_conv}.
\end{proof}


\section{Malliavin Calculus for the 2D Navier-Stokes equations in the vorticity formulation}
\label{Malliavin_section}

We can use the framework of the Malliavin calculus in the setting introduced in Section \ref{hyp_res}, namely the underlying Gaussian space on which to perform Malliavin calculus is given by the isonormal Gaussian process on the Hilbert space $\mathcal{H}_T$. We recall here some basic facts about the Malliavin calculus. For full details we refer to \cite{Nualart2006}.

A $\cF$-measurable real
valued random variable $F$ is said to be cylindrical if it can be
written as
\begin{equation*}
F=f \left( W(\phi^1) ,\ldots, W(\phi^n) \right)\;,
\end{equation*}
where $\phi^i \in \cH_T$ and $f:\bR^n \to \bR$ is a $C^{\infty}$ bounded function. The set of
cylindrical random variables is denoted by $\mathcal{S}$. The
Malliavin derivative of $F \in \mathcal{S}$ is the 
stochastic process $D F = \{D_{\sigma} F,\ \sigma \in \mathcal{H}_T\}$ given by
\[
{D} F=\sum_{i=1}^{n} \phi^i  \frac{\partial f}{\partial
x_i} \left( W(\phi^1) ,\ldots, W(\phi^n) \right).
\]
The operator $D$ is closable from
$\mathcal{S}$ into $L^p \left( \Omega , \cH_T\right)$. We denote by
$\bD^{1,p}(\mathcal{H}_T)$ the closure of the class of
cylindrical random variables with respect to the norm
\[
\left\| F\right\| _{1,p}=\left( \bE\left( |F|^{p}\right)
+\mathbb{E} \left\| D F\right\|
_{\mathcal{H}_T}^{p} \right) ^{\frac{1}{p}}.
\]
We also introduce the localized spaces; a random variable $F$ belongs to
$\bD^{1,p}_{\rm loc} (\mathcal{H}_T)$ if there exists a sequence of sets
$\Omega_n \subset \Omega$ and a sequence of random variables $F_n \in \bD^{1,p} 
(\cH_T)$ such that $\Omega_n \uparrow \Omega$
almost surely and $F = F_n$ on $\Omega_n$.
Then for any $n$ we set $DF=DF_n$ on $\Omega_n$. We refer to $(\Omega_n, F_n)$ as a localizing sequence for $F$.

The following key result stems from \cite[Theorem
2.1.3]{Nualart2006}:

\begin{theorem}\label{theo:dens}
Let $F$ be a $\cF$-measurable random
variable such that
$F\in \bD^{1,1}_{\rm loc}(\cH_T)$
and 
\begin{equation}
\label{44}
\|DF\|_{\mathcal{H}_T} >0, \quad \mathbb{P}-a.s.
\end{equation}
Then the law of $F$ has a density with respect to the Lebesgue
measure on $\bR$. 
\end{theorem}
In order to prove the assumption $F\in \bD^{1,1}_{\rm loc}(\cH_T)$ in our setting, we shall work on a sequence of smoothed processes and use the following result (see \cite[Lemma 1.5.3]{Nualart2006}).

\begin{proposition}
\label{p.3}
Let $\{F_k\}_k$ be a sequence of random variables in $\bD^{1,p}$ for some $p > 1$. 
Assume that the sequence $F_k$ converges to $F$ in
$L^p(\Omega)$ and that
\begin{equation}
\label{e4}
\sup_k \|F_k\|_{1,p} < \infty.
\end{equation}
Then $F$ belongs to $\bD^{1,p}$.
\end{proposition}
We use these results for the random variable $\xi(t,x)$, solution to equation \eqref{Walsh} and the random variable $\xi_N(t,x)$ solution to equation \eqref{Walsh_truncated}. More precisely, by means of Proposition \ref{p.3}, in Section \ref{Mall_trunc_eq}, we show that $\xi_N(t,x) \in \mathbb{D}^{1,p}$; hence $\xi(t,x) \in \mathbb{D}^{1,p}_{\text{loc}}$.
In Subsection \ref{non_deg} we prove that $\xi_N(t,x)$ satisfies assumption \eqref{44} of Theorem \ref{theo:dens}. The same condition holds for $\xi(t,x)$ as we shall see in Subsection \ref{density_subsec}.

\subsection{Malliavin analysis of the truncated equation}
\label{Mall_trunc_eq}
In order to show that $\xi_N(t,x) \in \mathbb{D}^{1,p}$ we use Proposition \ref{p.3}. We introduce a Picard approximation sequence $\left\{ \xi^k_{N}\right\}_k$ for $\xi_N$ and we show that as $k \rightarrow +\infty$, the sequence $\xi^k_{N}(t,x)$ converges to $\xi_N(t,x)$ in $L^p(\Omega)$ (for $N\ge1$ fixed) and $\sup_{k}\|\xi^k_{N}(t,x)\|_{1,p} < \infty$ uniformly in $(t,x)\in \left[0,T\right] \times D$. 
A similar argument has been used in \cite{Cardon-Weber2001} for the Cahn-Hilliard stochastic equation and in \cite{Morien1999} for the one dimensional Burgers equation.
Let us point out that the smoothness of the density cannot be obtained via this location argument, since this procedure does not provide the boundedness of the Malliavin derivatives of every order.

First, we need to improve the result of Proposition \ref{result_truncated}. This is done in the following theorem, whose proof provides the approximating sequence $\{\xi_N^k\}_k$ of the Picard scheme.

\begin{theorem}
\label{conv_Lp}
Fix $N\ge 1$ and $p>4$. If $b>0$ in \eqref{delta} and $\xi_0$ is a continuous function on $D$, then the solution process $\xi_N$ to 
\eqref{Walsh_truncated}
satisfies 
\begin{equation}
\label{p_mean}
\sup_{(t,x) \in \left[0,T\right] \times D} \mathbb{E}|\xi_N(t,x)|^p < \infty.
\end{equation}
\end{theorem}

\begin{proof}
Let us consider a Picard iteration scheme for equation \eqref{Walsh_truncated}. We define 
\begin{equation}
\label{pic1}
\xi^0_{N}(t,x)=\int_D\xi_0(y)g(t,x,y)\,{\rm d}y 
\end{equation}
and recursively, for $k \ge 0$
\begin{align}
\label{pic2}
\xi_{N}^{k+1}(t,x)&=\xi^0_{N}(t,x)+ z(t,x) +(Jq_N(\xi_N^k))(t,x)
\end{align}
with $z$ and $Jq_N(\xi_N^k)$ defined respectively as in \eqref{reg_z} and \eqref{defJfi}.
Notice that every term in \eqref{pic2} is well defined. The well posedness of the stochastic term follows from \eqref{conv_stoc}. On the other hand \eqref{continuity}, \eqref{stima_q} and Proposition \ref{result_truncated} provide the well posedness of the non linear term.

For every $(t,x) \in \left[0,T \right] \times D$, from Lemma \ref{regularization}\textit{(ii)} (for $\gamma=p$, provided $p>4$) and Lemma \ref{Lipschitz} we get

\begin{equation}
\label{50}
\mathbb{E}|\xi_{N}^{k+1}(t,x)-\xi_{N}^{k}(t,x)|^p \le C_{N,T,p} \int_0^t \int_D \mathbb{E}|\xi_{N}^{k}(s,y)-\xi_{N}^{k-1}(s,y)|^p\,{\rm d}y\,{\rm d}s.\end{equation}
For every $k \ge 1$ we set
\begin{equation*}
\varphi_k(t,x) = \mathbb{E}|\xi_{N}^{k}(t,x)-\xi_{N}^{k-1}(t,x)|^p.
\end{equation*}
Then $\varphi_1 \in L^1(\left[0,T\right] \times D)$ for the same considerations made for the well posedness of \eqref{pic2}. From \eqref{50}, by iteration (see e.g. \cite[Theorem 2.4.3]{Nualart2006}) and \cite[Proposition 5.1]{Morien1999}) we get
\begin{align*}
\varphi_{k+1}(t,x)& \le C_{N,p,T} \int_0^t \int_D \varphi_k(s,y)\, {\rm d}y\, {\rm d}s \le ... 
\\
&\le C_{N,p,T}^k \int_0^t \int_D \left[\int_0^{s_1} \int_D \cdot \cdot \cdot \left(\int_0^{s_{k-1}} \int_D \varphi_1(s_k,y_k)\,{\rm d}y_k\, {\rm d}s_k\right) \cdot \cdot \cdot \right]\,{\rm d}y_1 \, {\rm d}s_1
\\
&=C_{N,p,T}^k|D|^{k-1} \frac{t^{k-1}}{(k-1)!}\int_0^t \int_D \varphi_1(s_k,y_k)\,{\rm d}y_k \, {\rm d}s_k.
\\
\end{align*}
This allows us to infer that 
\begin{align*}
\sum_{k=1}^{\infty}\sup_{(t,x) \in \left[0,T \right] \times D} \varphi_{k+1}(t,x) \le \left( \sum_{k=1}^{\infty}C_{N,p,T}^k |D|^{k-1}\frac{t^{k-1}}{(k-1)!}\right)\int_0^t \int_D \varphi_1(s,y)\,{\rm d}y \, {\rm d}s.
\end{align*}
Since the latter series converges, we deduce that
\begin{equation*}
\sum_{k = 1}^{\infty} \sup_{(t,x) \in \left[0,T\right]\times D}\mathbb{E}|\xi_{N}^{k+1}(t,x)-\xi_{N}^{k}(t,x)|^p < \infty.
\end{equation*}
This implies that, as $k$ tends to infinity, the sequence $\xi_{N}^{k}(t,x)$ converges in $L^p(\Omega)$, uniformly in time and space, to a stochastic process $\xi_N(t,x)$. Moreover,
\begin{equation}
\label{p_mean_k}
\sup_k\sup_{(t,x) \in \left[0,T\right] \times D} \mathbb{E}|\xi_{N}^{k}(t,x)|^p < \infty.
\end{equation}
It follows that the process $\xi_N(t,x)$ is adapted and satisfies \eqref{Walsh_truncated} and \eqref{p_mean}.
\end{proof}



Let us study the Malliavin derivative of the solution $\xi_N$ to the smoothed equation \eqref{Walsh_truncated}.
Let us recall that the underlying Gaussian space on which to perform Malliavin calculus is given by the isonormal Gaussian process on the Hilbert space $\mathcal{H}_T:= L^2(0,T; L^2_Q)$ which can be associated to the noise coloured in space by the covariance $Q$. 

In this part, to keep things as simple as possible, in some points we go back to the notation involving $\xi_N$ and $v_N$ instead of $q_N(\xi_N)$, with $v_N=k\ast \xi_N$. Keeping in mind the definition of $\tilde q_N(\xi)$ given \eqref{q_tilde} we state the following result. 

\begin{theorem}
\label{D12}
Fix $N \ge 1$. Suppose that $b>0$ in \eqref{delta} and $\xi_0$ is a continuous function on $D$.
Then for all $(t,x) \in \left[0,T\right] \times D$ the solution $\xi_N(t,x)$ to \eqref{Walsh_truncated} belongs to $\mathbb{D}^{1,p}$ for every $p>4$ and its Malliavin derivative satisfies the equation
\begin{align}
\label{Malliavin_derivative}
&D_{r,z}\xi_N(t,x)=g(t-r, x,z) \pmb{1}_{\left[0,t\right]}(r) 
\notag\\
&+ \int_r^t\int_D \nabla_y g(t-s,x,y) \cdot \  v_N(s,y) \Theta_{N}(\|\xi_N(s, \cdot)\|_{L^{p}}) \ D_{r,z}\xi_N(s,y) {\rm d}y\, {\rm d} s
\notag\\
&+ \int_r^t\int_D \left( \nabla_yg(t-s,x,y) \cdot \int_D  k(y- \alpha) D_{r,z}\xi_N(s, \alpha) \, {\rm d}\alpha\right)
\notag\\
&\qquad \qquad \Theta_{N}(\|\xi_N(s, \cdot)\|_{L^{p}}) \xi_N(s,y)\,{\rm d}y \,{\rm d} s
\notag\\
&+ \int_r^t\int_D \nabla_y g(t-s,x,y) \cdot \ \tilde q_N(\xi_N(s,\cdot))(y) 
\notag\\
&  \qquad \qquad p\|\xi_N(s, \cdot)\|^{1-p}_{L^p}\left(\int_D|\xi_N(s, \beta)|^{p-2}\xi_N(s, \beta)D_{r,z}\xi_N(s, \beta)\, {\rm d} \beta \right){\rm d}y\, {\rm d} s
\end{align}
if $ r \le t$, and $D_{r,z}\xi_N(t,x)=0$ if $r>t$.
\end{theorem}

\begin{proof}

The proof of this part is based on Proposition \ref{p.3}. Let us consider the Picard approximation sequence $\left\{\xi_{N}^{k}(t,x)\right\}_k$ defined in \eqref{pic1}-\eqref{pic2}; given the convergence (as $k \rightarrow +\infty$) obtained in the proof of Theorem \ref{conv_Lp}, it is sufficient to show that
\begin{equation}
\label{e5}
\sup_k \sup_{(t,x) \in [0,T] \times D} \bE \|D \xi_{N}^{k}(t,x)\|^{p}_{\cH_T} < +\infty,
\end{equation}
in order to prove that $\xi_N(t,x) \in \mathbb{D}^{1,p}$.
Since $\xi_{N}^{0}$ is deterministic, it belongs to $\mathbb{D}^{1,p}$ and its Malliavin derivative is zero.
Let us suppose that, for $k \ge 1$ and $p>4$, $\xi_{N}^{k}(t,x)\in \mathbb{D}^{1,p}$ for every $(t,x) \in \left[0,T\right] \times D$ and 
\begin{equation*}
\sup_{(t,x) \in \left[0,T\right] \times D} \mathbb{E} \|D\xi_{N}^{k}(t,x)\|^p_{\mathcal{H}_T} < \infty.
\end{equation*}
Applying the operator $D$ to equation \eqref{pic2} we obtain 
that the Malliavin derivative of $\xi_N^k(t,x)$ satisfies the equation
(for more details see for instance \cite[Proposition 2.15 and Proposition 2.16]{da2008introduction} and \cite[Proposition 1.3.2]{Nualart2006})
\begin{align}
\label{mal_pic}
&D_{r,z}\xi_{N}^{k+1}(t,x)=g(t-r, x,z) \pmb{1}_{\left[0,t\right]}(r) 
\notag\\
&+ \int_r^t\int_D \nabla_y g(t-s,x,y) \cdot \  v_{N}^{k}(s,y) \Theta_{N}(\|\xi_{N}^{k}(s, \cdot)\|_{L^{p}}) \ D_{r,z}\xi_{N}^{k}(s,y) {\rm d}y\, {\rm d} s
\notag\\
&+ \int_r^t\int_D \left( \nabla_yg(t-s,x,y) \cdot \int_D  k(y- \alpha) D_{r,z}\xi_{N}^{k}(s, \alpha) \, {\rm d}\alpha\right)
\notag\\
& \qquad \qquad  \Theta_{N}(\|\xi_{N}^{k}(s, \cdot)\|_{L^{p}}) \xi_{N}^{k}(s,y)\, {\rm d}y\, {\rm d} s
\notag\\
&+ \int_r^t\int_D \nabla_y g(t-s,x,y) \cdot \  \tilde q_N(\xi_{N}^{k}(s,\cdot))(y)
\notag\\
& \qquad \qquad p\|\xi_{N}^{k}(s, \cdot)\|^{1-p}_{L^p}\left(\int_D|\xi_{N}^{k}(s, \beta)|^{p-2} \xi_{N}^{k}(s, \beta)D_{r,z}\xi_{N}^{k}(s, \beta)\, {\rm d} \beta \right){\rm d}y \,{\rm d} s.
\end{align}
We analyze the three integrals in the r.h.s.
Let us set for simplicity
\begin{equation}
\label{I1}
I_1(r,z):=\int_r^t\int_D \nabla_y g(t-s,x,y) \cdot \  v_{N}^{k}(s,y) \Theta_{N}(\|\xi_{N}^{k}(s, \cdot)\|_{L^{p}}) \ D_{r,z}\xi_{N}^{k}(s,y) {\rm d}y\, {\rm d} s
\end{equation}
\begin{multline}
\label{I2}
I_2(r,z):=\int_r^t\int_D \left( \nabla_yg(t-s,x,y) \cdot \int_D  k(y- \alpha) D_{r,z}\xi_{N}^{k}(s, \alpha) \, {\rm d}\alpha\right) 
\\
\Theta_{N}(\|\xi_{N}^{k}(s, \cdot)\|_{L^{p}}) \xi_{N}^{k}(s,y)\, {\rm d}y\, {\rm d} s
\end{multline}
\begin{multline}
\label{I3}
I_3(r,z):= \int_r^t\int_D \nabla_y g(t-s,x,y) \cdot \ \tilde q_N(\xi_{N}^{k}(s,\cdot))(y)
 \\
\qquad \left(p\|\xi_{N}^{k}(s, \cdot)\|^{1-p}_{L^p}\int_D|\xi_{N}^{k}(s, \beta)|^{p-2} \xi_{N}^{k}(s, \beta)D_{r,z}\xi_{N}^{k}(s, \beta)\, {\rm d} \beta \right)\,{\rm d}y \,{\rm d} s.
 \end{multline}
Then
\begin{align}
\label{p-est}
\mathbb{E}\|D\xi_{N}^{k+1}(t,x)\|^{p}_{\mathcal{H}_T} \le C_p \left(\|g(t-\cdot,x,\cdot){\pmb 1}_{[0,t]}(\cdot) \|^{p}_{\cH_T} 
+ \sum_{i=1}^3\mathbb{E}\|I_i\|^{p}_{\mathcal{H}_T} \right).
\end{align}
Let us estimate the various terms in (\ref{p-est}). 
By the definition of $\mathcal{H}_T$ and 
\eqref{conv_stoc}, we get
\begin{align*}
\|g(t-\cdot,x,\cdot){\pmb 1}_{[0,t]}(\cdot) \|^{p}_{\cH_T} 
&=\left(\|g(t-\cdot,x,\cdot)\|^{2}_{\cH_t}\right)^{\frac p2}  
\le \left( \frac {1}{2(2\pi)^2}\sum_{k \in \mathbb{Z}_0^2}  |k|^{-2-2b}\right)^{\frac p2}
< \infty.
\end{align*}

Minkowski's and H\"older's inequalities imply that
\begin{align*}
\mathbb{E}\|I_1&\|^{p}_{\mathcal{H}_T}
\le \mathbb{E} \left[\int_0^t\int_D \left|\nabla_yg(t-s,x,y) \cdot \  v_{N}^{k}(s,y) \Theta_{N}(\|\xi_{N}^{k}(s, \cdot)\|_{L^{p}})\right| \ \|D\xi_{N}^{k}(s,y) \|_{\mathcal{H}_T}{\rm d}y {\rm d} s \right]^{p}
\\
&\le \mathbb{E} \left[\int_0^t  |\Theta_{N}(\|\xi_{N}^{k}(s, \cdot)\|_{L^{p}})| \|\nabla_yg(t-s,x,\cdot)\|_{L^{\frac{p}{p-1}}} \right.
\\
&\left. \qquad \qquad \qquad \left(\int_D  \ | v_{N}^{k}(s,y)|^p \ \|D\xi_{N}^{k}(s,y) \|^p_{\mathcal{H}_T}{\rm d}y \right)^{\frac1p}{\rm d} s \right]^{p}
\\
&\le\mathbb{E} \left[\int_0^t |\Theta_{N}(\|\xi_{N}^{k}(s, \cdot)\|_{L^{p}})| \|\nabla_yg(t-s,x,\cdot)\|_{L^{\frac{p}{p-1}}} \right.
\\
&\left. \qquad \qquad \qquad  \|v_{N}^{k}(s, \cdot)\|_{L_{\infty}} \ \|D\xi_{N}^{k}(s,\cdot) \|_{L^{p}(D;\mathcal{H}_T)} \, {\rm d} s \right]^{p}
\\
&\le  C_N\left(\int_0^t \int_D |\nabla_y g(t-s,x,y)|^{\frac{p}{p-1}} \, {\rm d}y  \,{\rm d} s\right)^{p-1} 
\\
&\qquad \qquad \qquad \mathbb{E}\left[\int_0^t \|D\xi_{N}^{k}(s, \cdot)\|^{p}_{L^{p}(D;\mathcal{H}_T)} \, {\rm d}s\right]\ \qquad \text{by \eqref{L_infty}} 
\\
&\le C_{N} t^{\frac{p}{2}-2}\int_0^t \int_D \mathbb{E} \|D\xi_{N}^{k}(s,y)\|^{p}_{\mathcal{H}_T}\, {\rm d}y \, {\rm d}s \qquad \text{by \eqref{beta2} provided $p>4$.} 
\end{align*}
As regards the term $I_2$ using Fubini's Theorem, Minkowski's and H\"older's inequalities we have

\begin{align*}
\mathbb{E}\|I_2\|&^{p}_{\mathcal{H}_T}
=\mathbb{E}\left \Vert \int_r^t\int_D \left(  \int_D\nabla_yg(t-s,x,y) \cdot   k(y- \alpha)  \xi_{N}^{k}(s,y) \, {\rm d}y \right) \right.
\\
&\left. \qquad \qquad \qquad \Theta_{N}(\|\xi_{N}^{k}(s, \cdot)\|_{L^p})D\xi_N^k(s, \alpha)\, {\rm d}\alpha\, {\rm d}s \right\Vert^{p}_{\mathcal{H}_T}
\\
&\le \mathbb{E}\left[ \int_0^t\int_D \left| \left(  \int_D\nabla_yg(t-s,x,y) \cdot  k(y- \alpha)  \xi_{N}^{k}(s,y) \, {\rm d}y \right) \right. \right.
\\
&\left. \left. \qquad \qquad \qquad\Theta_{N}(\|\xi_{N}^{k}(s, \cdot)\|_{L^{p}}) \right| \|D\xi_{N}^{k}(s,\alpha)\|_{\mathcal{H}_T} \, {\rm d}\alpha\, {\rm d} s\right]^{p}
\\
&\le \mathbb{E}\left[ \int_0^t\int_D       \|\nabla_yg(t-s,x,\cdot) \cdot k(\cdot- \alpha)\|_{L^{\frac{p}{p-1}}}\|\xi_{N}^{k}(s, \cdot)\|_{L^{p}} \right.
\\
& \left. \qquad \qquad \qquad |\Theta_{N}(\|\xi_{N}^{k}(s, \cdot)\|_{L^{p}})|\           \|D\xi_{N}^{k}(s,\alpha)\|_{\mathcal{H}_T} \, {\rm d}\alpha\, {\rm d} s\right]^{p}
\\
&\le C_N\mathbb{E}\left[ \int_0^t\int_D     \|\nabla_yg(t-s,x,\cdot) \cdot k(\cdot- \alpha)\|_{L^{\frac{p}{p-1}}}     \|D\xi_{N}^{k}(s,\alpha)\|_{\mathcal{H}_T} \, {\rm d}\alpha\, {\rm d} s\right]^{p}
\\
&\le C_N\mathbb{E}\left[\int_0^t \left(\int_D \int_D|\nabla_yg(t-s,x,y) \cdot k(y- \alpha)|^{\frac{p}{p-1}}\, {\rm d}y\, {\rm d}\alpha\right)^{\frac{p-1}{p}} \right.
\\
&\left. \qquad \qquad \qquad \|D\xi_{N}^{k}(s, \cdot)\|_{L^p(D;\mathcal{H}_T)}\, {\rm d}s\right]^{p}.
\end{align*}
By means of Fubini's Theorem, if $p>4$, we can estimate the inner integral
\begin{align}
\label{star2}
\int_D \int_D|&\nabla_yg(t-s,x,y) \cdot k(y- \alpha)|^{\frac{p}{p-1}}\, {\rm d}y\, {\rm d}\alpha
\\
&\le\int_D \int_D|\nabla_yg(t-s,x,y)|^{\frac{p}{p-1}} \ |  k(y- \alpha)|^{\frac{p}{p-1}}\, {\rm d}y  \,{\rm d} \alpha 
\notag\\
&= \int_D |\nabla_yg(t-s,x,y)|^{\frac{p}{p-1}} \ \left(\int_D| k(y- \alpha)|^{\frac{p}{p-1}}\,{\rm d} \alpha\right)\, {\rm d}y  
\notag\\
&\le C\int_D|\nabla_yg(t-s,x,y)|^{\frac{p}{p-1}}\,{\rm d}y 
\qquad \text{by Lemma \ref{stimaBS} and Remark \ref{p_int_BS}}
\notag \\
&\le C_p (t-s)^{-\frac32\left(\frac{p}{p-1}\right)+1} \qquad \text{by \eqref{beta1}},
 \end{align}
obtaining 
\begin{align*}
\mathbb{E}\|I_2\|^{p}_{\mathcal{H}_T}
&\le C_{N ,p}\mathbb{E}\left[ \int_0^t (t-s)^{-\frac{p+2}{2p}}\|D\xi_{N}^{k}(s, \cdot)\|_{L^p(D;\mathcal{H}_T)} \, {\rm d}s \right]^{p}
\\
&\le C_{N,p} \left( \int_0^t (t-s)^{\frac{p+2}{2(1-p)}}\, {\rm d}s\right)^{p-1}\mathbb{E}\left[\int_0^t\|D\xi_{N}^{k}(s, \cdot)\|^{p}_{L^p(D;\mathcal{H}_T)}\, {\rm d}s \right]
\\
&\le C_{N,p}t^{\frac p2-2}\int_0^t\int_D\mathbb{E}\|D\xi_{N}^{k}(s, y)\|^{p}_{\mathcal{H}_T}\, {\rm d} y\, {\rm d}s,
\end{align*}
provided $p>4$.

For the last term $I_3$, using as above Minkowski's and H\"older's inequalities, we have
\begin{align*}
\mathbb{E}\|I_3\|^{p}_{\mathcal{H}_T}
&= \mathbb{E} \left \Vert\int_r^t\int_D \nabla_y g(t-s,x,y) \cdot \tilde q_N(\xi_N^k)(s,\cdot))(y)
\right.
\\
&\left. \qquad  p\|\xi_{N}^{k}(s, \cdot)\|^{1-p}_{L^p}\left(\int_D|\xi_{N}^{k}(s, \beta)|^{p-2}\xi_{N}^{k}(s, \beta)D\xi_{N}^{k}(s, \beta)\, {\rm d} \beta \right)\,{\rm d}y \, {\rm d} s \right \Vert_{\mathcal{H}_T}^{p}
\\
&\le \mathbb{E} \left[ 
\int_0^t \int_D |\nabla_yg(t-s,x,y)\cdot \tilde q_N(\xi_N^k)(s,\cdot))(y)| \right.
\\
&\left. \qquad \qquad p\|\xi_{N}^{k}(s, \cdot)\|^{1-p}_{L^p}\left(\int_D|\xi_{N}^{k}(s, \beta)|^{p-1}\|D\xi_{N}^{k}(s, \beta)\|_{\mathcal{H}_T}\, {\rm d} \beta \right)\,{\rm d}y \, {\rm d} s \right]^p
\\
&\le \mathbb{E} \left[ \int_0^t p\|\nabla_yg(t-s,x,\cdot) \cdot \tilde q_N(\xi_N^k(s, \cdot))\|_{L^1}\|D\xi_N^k(s, \cdot)\|_{L^p(D;\mathcal{H}_T)}\, {\rm d}s \right]^p.
\end{align*}
\eqref{beta1} and \eqref{stima_q_2} imply that 
\begin{align*}
\|\nabla_yg(t-s,x,\cdot) \cdot \tilde q_N(\xi_N^k(s, \cdot))\|_{L^1}
&\le \|\nabla_yg(t-s,x,\cdot)\|_{L^{\frac{p}{p-1}}} \|\tilde q_N(\xi_N^k(s, \cdot))\|_{L^p}
\\
&\le C_p (N+1)^2(t-s)^{-\frac{p+2}{2p}} \qquad \text{provided $p>4$}.
\end{align*}
Thanks to H\"older's inequality,
\begin{align*}
\mathbb{E}\|I_3\|^p_{\mathcal{H}_T} \le C_{N,p} \ \mathbb{E} \left[\int_0^t(t-s)^{-\frac{p+2}{2p}}\|D\xi_N^k(s, \cdot)\|_{L^p(D; \mathcal{H}_T)}\, {\rm d}s \right]^p
\\
\le C_{N,p} t^{\frac p2 - 2} \int_0^t \int_D \mathbb{E} \|D\xi_N^k(s,y)\|^p_{\mathcal{H}_T}\, {\rm d}y\, {\rm d}s
\end{align*}
provided $p>4$.

From the above estimates, if $p>4$, we infer
\begin{equation}
\label{5.15}
\mathbb{E} \|D\xi_{N}^{k+1}(t,x)\|^{p}_{\mathcal{H}_T} \le C_p +C_{N,T,p} \int_0^t \int_D \mathbb{E} \|D\xi_{N}^{k}(s,y)\|^{p}_{\mathcal{H}_T} \, {\rm d}y \,{\rm d}s.
\end{equation}
This proves that if $\xi_N^k(t,x) \in \mathbb{D}^{1,p}$, then $\xi_N^{k+1}(t,x) \in \mathbb{D}^{1,p}$. Moreover, iterating inequality \eqref{5.15} (which holds for every $(t, x) \in [0,T] \times D$ ) and proceeding as in the proof of Theorem \ref{conv_Lp}, we obtain \eqref{e5}. What remains to prove is equality \eqref{Malliavin_derivative}; but this is obtained by applying the operator $D$ to both members of equation \eqref{Walsh_truncated}.
\end{proof}

\subsection{Nondegeneracy condition}
\label{non_deg}
Now we check condition \eqref{44} for the solution $\xi_N$ to the truncated equation.
Let $t \in \left[0,T\right]$ and $x \in D$. We aim at proving that 
\begin{equation}
\label{58}
\|D\xi_N(t,x)\|^2_{\mathcal{H}_T} >0  \qquad \mathbb{P}-a.s.
\end{equation}
The following lemma is an improvement of Theorem \ref{D12} and it is needed in order to prove Theorem \ref{dens}.
We need to consider a time interval smaller than $\left[0,T\right]$ and consider the $\mathcal{H}_T$-norm of $\xi_N(t,x)$ on $(t- \varepsilon, t)$ for some $\varepsilon >0$ small enough. For every $\varphi \in \mathcal{H}_T$ we define the norm 
\begin{equation*}
\|\varphi\|_{\mathcal{H}_{(t-\varepsilon,t)}}:=\|\pmb{1}_{(t-\varepsilon,t)}(\cdot)\varphi\|_{\mathcal{H}_T}.
\end{equation*}  
It is straightforward to get
\begin{equation*}
\|\varphi\|_{\mathcal{H}_T} \ge \|\varphi\|_{\mathcal{H}_{(t-\varepsilon,t)}}.
\end{equation*}

\begin{lemma}
\label{D_eps}
Let $N\ge 1$, $b>1$ in \eqref{delta} and $p>4$. If $\xi_0$ is a continuous function on $D$, then there exists a constant $C_{N,p,Q,T}$ such that for every $0<\varepsilon<t$
\begin{equation*}
\sup_{\sigma \in \left[t- \varepsilon, t\right]} \sup_{x \in D} \mathbb{E}\|D\xi_N(\sigma, x)\|^p_{\mathcal{H}(t- \varepsilon, t)} \le C_{N,p,Q,T} \varepsilon^{\frac p2}.
\end{equation*}
\end{lemma}

\begin{proof}
For $t-\varepsilon \le \sigma \le t$, set $\eta^{\varepsilon}_N(\sigma, x)= \mathbb{E}\|D\xi_N(\sigma, x)\|^p_{\mathcal{H}_{(t-\varepsilon, \sigma)}}$.
According to \eqref{Malliavin_derivative},
\begin{equation*}
\eta^{\varepsilon}_N(\sigma, x)\le C_p \left( \|g(\sigma-\cdot,x, \cdot)\pmb{1}_{\left[0,\sigma\right]}(\cdot)\|^p_{\mathcal{H}_{(t-\varepsilon, \sigma)}}+ \sum_{i=1}^3 \mathbb{E}\|I_i\|^p_{\mathcal{H}_{(t-\varepsilon, \sigma)}}\right),
\end{equation*}
where the terms $I_i$, $i=1,2,3$, are defined in \eqref{I1}-\eqref{I3}.
By \eqref{conv_stoc} and the change of variables $s=r-\sigma+\varepsilon$, we get
\begin{align}
\label{60}
\int_{t- \varepsilon}^{\sigma}\|&g(\sigma-r, x, \cdot)\|^2_{L^2_Q}\, {\rm d}r
\le\sum_{k \in \mathbb{Z}_0^2}|k|^{-2b}\int_0^{\varepsilon} e^{-2|k|^2(\varepsilon-s)}|e_k(x)|^2 \, {\rm d}s
\notag\\
&=\frac{1}{(2\pi)^2}\sum_{k \in \mathbb{Z}_0^2} \frac{|k|^{-2b-2}}{2}(1-e^{-2|k|^2\varepsilon})
\le \frac{1}{(2 \pi)^2}\sum_{k \in \mathbb{Z}_0^2} \frac{|k|^{-2b-2}}{2}(2|k|^2\varepsilon)
\notag \\
&=\frac{\varepsilon}{(2\pi)^2} \sum_{k \in \mathbb{Z}_0^2}|k|^{-2b} 
= \frac{\varepsilon}{(2 \pi)^2}\  \text{Tr}Q,
\end{align}
which is finite provided $b>1$.
So
\begin{align}
\label{g_eps}
\mathbb{E} \| g(\sigma-\cdot,x,\cdot) {\pmb 1}_{[0,\sigma]}(\cdot) \|^p_{\mathcal{H}_{\left(t-\varepsilon,\sigma \right)}} 
&=\left(\int_{t- \varepsilon}^{\sigma}\|g(\sigma-r, x, \cdot)\|^2_{L^2_Q}\,{\rm d}r\right)^{\frac p2}
\notag\\
&\le\frac{ \left( \text{Tr} Q\right)^{\frac p2}\varepsilon^{\frac p2}}{(2 \pi)^p}=C_{p,Q}\ \varepsilon^{\frac p2}.
\end{align}

Minkowski's and H\"older's inequalities and \eqref{L_infty} imply that 
\begin{align*}
\mathbb{E}\|I_1\|^p_{\mathcal{H}(t- \varepsilon, \sigma)} 
&= \mathbb{E} \left[ \int_{t- \varepsilon}^{\sigma}\left\Vert\int_r^{\sigma}\int_D\nabla_yg(t-s,x,y) \cdot v_N(s,y)\right. \right.
\\
&\left. \left. \qquad \qquad \Theta_N(\|\xi_N(s, \cdot)\|_{L^p})D_{r,\cdot}\xi_N(s,y)\, {\rm d}y\, {\rm d}s
\right\Vert^2_{L^2_Q}\, {\rm d}r\right]^{\frac p2}
\\
&= \mathbb{E} \left[ \int_{t- \varepsilon}^{\sigma}\left\Vert\int_{t-\varepsilon}^{\sigma}\int_D\nabla_yg(t-s,x,y) \cdot v_N(s,y)\right. \right.
\\
&\left. \left. \qquad \qquad \Theta_N(\|\xi_N(s, \cdot)\|_{L^p})D_{r,\cdot}\xi_N(s,y)\, {\rm d}y\, {\rm d}s
\right\Vert^2_{L^2_Q}\, {\rm d}r\right]^{\frac p2}
\\
&\le \mathbb{E} \left[\int_{t-\varepsilon}^{\sigma}\int_D \left|\nabla_yg(t-s,x,y) \cdot v_N(s,y)\right| \right.
\\
& \left. \qquad \qquad\left| \Theta_N(\|\xi_N(s, \cdot)\|_{L^p})\right| \ \|D\xi_N(s,y)\|_{\mathcal{H}_{(t-\varepsilon,\sigma)}}\, {\rm d}y \, {\rm d}s\right]^p
\\
&\le C_N\left(\int_0^T \int_D|\nabla_yg(t-s,x,y)|^{\frac{p}{p-1}}\, {\rm d}y \, {\rm d}s\right)^{p-1} 
\\
&\qquad \qquad  \int_{t-\varepsilon}^{\sigma}\int_D \mathbb{E}\|D\xi_N(s,y)\|^p_{\mathcal{H}_{(t-\varepsilon, s)}}\, {\rm d}y \, {\rm d}s
\\
&\le C_N T^{\frac p2-2} \int_{t-\varepsilon}^{\sigma}\sup_{y \in D} \mathbb{E}\|D\xi_N(s,y)\|^p_{\mathcal{H}_{(t-\varepsilon, s)}}\, {\rm d}s
\qquad\text{ by \eqref{beta2} if } p>4.
\end{align*}
As regards the term $I_2$, proceeding in a similar way, by means of Fubini Theorem, H\"older's and Minkowski's inequalities we get
\begin{align*}
\mathbb{E}&\|I_2\|^p_{\mathcal{H}_{(t-\varepsilon,\sigma)}}
\\
&= \mathbb{E}\left[ \int_{t-\varepsilon}^{\sigma}\left\Vert \int_{t-\varepsilon}^{\sigma}\int_D\left( \nabla_yg(t-s,x,y)\cdot \int_Dk(y-\alpha)D_{r,\cdot}\xi_N(s,\alpha)\, {\rm d}\alpha\right)\right. \right.
\\
& \left. \left. \qquad \qquad \qquad \Theta_N(\|\xi_N(s, \cdot)\|_{L^p})\xi_N(s,y)\, {\rm d}y \, {\rm d}s\right\Vert^2_{L^2_Q}\, {\rm d}r\right]^{\frac p2}
\\
&\le \mathbb{E}\left[ \int_{t-\varepsilon}^{\sigma}\int_D\left|\int_D \nabla_yg(t-s,x,y)\cdot k(y-\alpha)\xi_N(s,y)\Theta_N(\|\xi_N(s, \cdot)\|_{L^p})\, {\rm d}y \right| \right.
\\
& \left. \qquad \qquad \qquad \|D\xi_N(s, \alpha)\|_{\mathcal{H}_{(t-\varepsilon, \sigma)}}\, {\rm d}\alpha\, {\rm d}s\right]^p
\\
&\le C_N\mathbb{E} \left[ \int_{t-\varepsilon}^{\sigma}\int_D\|\nabla_yg(t-s, x, \cdot) \cdot k(\cdot-\alpha)\|_{L^{\frac{p}{p-1}}}\ \|D\xi_N(s, \alpha)\|_{\mathcal{H}_{(t-\varepsilon, \sigma)}}\, {\rm d}\alpha\, {\rm d}s\right]^p
\\
&\le C_N \left( \int_0^T (t-s)^{\frac{p+2}{2(1-p)}}\, {\rm d}s \right)^{p-1} 
\\
& \qquad \qquad \qquad \int_{t-\varepsilon}^{\sigma}\int_D \mathbb{E} \|D\xi_N(s, y)\|^p_{\mathcal{H}_{(t-\varepsilon, s)}}\,  {\rm d}y\, {\rm d}s
\qquad \text{ by \eqref{star2} if } p>4
\\
&\le C_N T^{\frac p2 -2} \int_{t-\varepsilon}^{\sigma}\sup_{y \in D} \mathbb{E} \|D\xi_N(s, y)\|^p_{\mathcal{H}_{(t-\varepsilon, s)}}\, {\rm d}s.\end{align*}

For the last term $I_3$, Minkowski's any H\"older's inequalities imply that
\begin{align*}
\mathbb{E}&\|I_3\|^p_{\mathcal{H}_{(t-\varepsilon, \sigma)}} 
=\mathbb{E} \left[
\int_{t - \varepsilon}^{\sigma} 
\left\Vert p\int_{t-\varepsilon}^{\sigma} \int_D \nabla_yg(t-s,x,y)\cdot \tilde q_N(\xi_N(s, \cdot))(y)\|\xi_N(s, \cdot)\|^{1-p}_{L^p} \right. \right.
\\
&\left. \left.\qquad \qquad \qquad \left( \int_D |\xi_N(s, \beta)|^{p-2}\xi_N(s, \beta)D_{r, \cdot}\xi_N(s, \beta)\, {\rm d}\beta\right)\, {\rm d}y \, {\rm d}s
\right\Vert^2_{L^2_Q}\, {\rm d}r
\right]^{\frac p2}
\\
& \le \mathbb{E} \left[ 
\int_{t-\varepsilon}^{\sigma}\int_D p |\nabla_yg(t-s,x,y)\cdot \tilde q_N(\xi_N(s, \cdot))(y)| \ \|\xi_N(s, \cdot)\|^{1-p}_{L^p} \right.
\\
& \left. \qquad \qquad \qquad \left(\int_D |\xi_N(s, \beta)|^{p-1} \|D\xi_N(s, \beta)\|_{\mathcal{H}_{(t-\varepsilon,\sigma)}}\, {\rm d}\beta
\right)\, {\rm d}y \, {\rm d}s 
\right]^p
\\
&\le C_{N,p}\mathbb{E}\left[\int_{t-\varepsilon}^{\sigma}p\|\nabla_yg(t-s,x,\cdot)\|_{L^{\frac{p}{p-1}}} \ \|D\xi_N(s, \cdot)\|_{L^p(D;\mathcal{H}_{(t-\varepsilon, \sigma)})} \right]^p\ \text{ by \eqref{stima_q_2}}
\\
&\le C_{N,p}\left(\int_0^T \int_D|\nabla_yg(t-s,x,y)|^{\frac{p}{p-1}}\, {\rm d}y \, {\rm d}s\right)^{p-1} \\
& \qquad \qquad \qquad \int_{t-\varepsilon}^{\sigma}\int_D \mathbb{E}\|D\xi_N(s,y)\|^p_{\mathcal{H}_{(t-\varepsilon, s)}}\, {\rm d}y \, {\rm d}s
\\
&\le C_{N,p} T^{\frac p2-2} \int_{t-\varepsilon}^{\sigma}\sup_{y \in D} \mathbb{E}\|D\xi_N(s,y)\|^p_{\mathcal{H}_{(t-\varepsilon, s)}} \, {\rm d}s
\qquad\text{ by \eqref{beta2} if } p>4.
\end{align*}
Collecting the above estimates we get 
\begin{align*}
\sup_{x \in D}\eta^{\varepsilon}_N(\sigma, x) \le C_{p,Q} \varepsilon^{\frac p2} + C_{N,p,T} \int_{t-\varepsilon}^{\sigma}\sup_{y \in D}\eta_N^{\varepsilon}(s,y)\, {\rm d}s \qquad   \text{for every}\ \sigma \in \left[t-\varepsilon,t\right].
\end{align*}
By the Gronwall's lemma it follows 
\begin{equation*}
\sup_{x \in D}\eta^{\varepsilon}_N(\sigma, x) \le C_{p,Q} \varepsilon^{\frac p2} e^{C_{N,p,T}(\sigma-t+\varepsilon)} 
\le C_{p,Q} \varepsilon^{\frac p2} e^{C_{N,p,T}T} \qquad \text{for every}\ \sigma \in \left[t-\varepsilon,t\right].
\end{equation*} 
Since for $\sigma \in \left[t-\varepsilon,t\right]$, $\|D\xi_N(\sigma,x)\|^p_{\mathcal{H}_{(t-\varepsilon, \sigma)}}=\|D\xi_N(\sigma,x)\|^p_{\mathcal{H}_{(t-\varepsilon, t)}}$ we finally get
\begin{equation*}
\sup_{\sigma \in \left[t- \varepsilon, t\right]} \sup_{x \in D} \mathbb{E}\|D\xi_N(\sigma, x)\|^p_{\mathcal{H}_{(t- \varepsilon, t)}} \le C_{N,p,Q,T} \varepsilon^{\frac p2}.
\end{equation*}

\end{proof}

\begin{theorem}
\label{dens}
Suppose $b>1$ in \eqref{delta} and assume that $\xi_0$ is a continuous function on $D$. Then, for every $t \in \left[0,T\right]$ and $x \in D$, the image law of the random variable $\xi_N(t,x)$ is absolutely continuous with respect to the Lebesgue measure on $\mathbb{R}$.
\end{theorem}

\begin{proof}
In order to prove that 
$\|D \xi_N(t,x)\|^2_{\mathcal{H}_T} >0 \qquad \mathbb{P}-a.s.$ we will show that
$$\mathbb{P}(\|D\xi_N(t,x)\|^2_{\mathcal{H}_T}=0)=0,$$
or, better, that
\begin{equation}
\mathbb{P}(\|D\xi_N(t,x)\|^2_{\mathcal{H}_T}< \delta)\rightarrow 0 \qquad \text{as} \ \ \delta \rightarrow 0.
\end{equation}
Let us fix $\varepsilon >0$ sufficiently small, according to \eqref{Malliavin_derivative}, by means of the inequality $(a+b)^2 \ge \frac12 a^2-b^2$, we get
\begin{align*}
\|D \xi_N(t,x)\|^2_{\mathcal{H}_T} 
&= \int_0^T\|D_{r, \cdot} \xi_N(t,x)\|^2_{L^2_Q}\, {\rm d}r  
\ge \int_{t- \varepsilon}^t\|D_{r, \cdot} \xi_N(t,x)\|^2_{L^2_Q}\, {\rm d}r  
\\
&=\int_{t- \varepsilon}^t \left\Vert g(t-r, x,\cdot) \pmb{1}_{\left[0,t\right]}(r) +\sum_{i=1}^3I_i(r,\cdot) \right\Vert_{L^2_Q}^2 \, {\rm d}r  
 \\
 &\ge \frac 12 \int_{t- \varepsilon}^t \| g(t-r, x,\cdot) \|^2_{L^2_Q} \, {\rm d}r - \int_{t-\varepsilon}^t \left\Vert \sum_{i=1}^3I_i(r,\cdot)\right\Vert_{L^2_Q}^2 \, {\rm d}r,
\end{align*}
where the terms $I_i$ are defined in \eqref{I1}-\eqref{I3}.
Let us set for simplicity
\begin{equation*}
I(t,x, \varepsilon)=\int_{t-\varepsilon}^t \left\Vert \sum_{i=1}^3I_i(r,\cdot) \right\Vert_{L^2_Q}^2 \, {\rm d}r,
 \qquad 
A(x, \varepsilon)= \int_{t- \varepsilon}^t \| g(t-r, x,\cdot) \|^2_{L^2_Q} \, {\rm d}r.
\end{equation*}
By means of Chebyschev's inequality, for $\delta>0$ sufficiently small, 
we have
\begin{equation}
\label{Chebyc}
\mathbb{P}(\|D\xi_N(t,x)\|^2_{\mathcal{H}_T}< \delta) \le \mathbb{P}(I(t,x,\varepsilon) \ge \frac 12 A(x, \varepsilon)-\delta) \le \frac{\mathbb{E}|I(t,x,\varepsilon)|^{\frac p2}}{\left(\frac12A(x, \varepsilon)- \delta\right)^{\frac p2}}.
\end{equation}
Let us find an upper estimate for $\mathbb{E}|I(t,x,\varepsilon)|^{\frac p2} \le C_p \sum_{i=1}^3 \mathbb{E} \left|\int_{t-\varepsilon}^t \|I_i(r, \cdot)\|^2_{L^2_Q}\, {\rm d}r\right|^{\frac p2}$.

Minkowski's and H\"older's inequalities and \eqref{L_infty} imply that 
\begin{align*}
\mathbb{E}&
\left | \int_{t-\varepsilon}^t \|I_1(r, \cdot)\|^2_{L^2_Q}\, {\rm d}r \right |^{\frac p2}
\\
&= \mathbb{E} \left[ \int_{t- \varepsilon}^{t}\left\Vert\int_{t-\varepsilon}^{t}\int_D\nabla_yg(t-s,x,y) \cdot v_N(s,y) \right. \right.
\\
&\left. \left. \qquad \qquad \qquad \Theta_N(\|\xi_N(s, \cdot)\|_{L^p})D_{r,\cdot}\xi_N(s,y)\, {\rm d}y\, {\rm d}s
\right\Vert^2_{L^2_Q}\, {\rm d}r\right]^{\frac p2}
\\
&\le C_N\left(\int_{t-\varepsilon}^t \int_D|\nabla_yg(t-s,x,y)|^{\frac{p}{p-1}}\, {\rm d}y \, {\rm d}s\right)^{p-1}  \int_{t-\varepsilon}^{t}\int_D \mathbb{E}\|D\xi_N(s,y)\|^p_{\mathcal{H}_{(t-\varepsilon, t)}}\, {\rm d}y \, {\rm d}s
\end{align*}
Using Lemma \ref{D_eps} with $t-\varepsilon \le s \le t$ and \eqref{beta2}, provided $p>4$, we deduce that
\begin{align*}
\mathbb{E}
\left | \int_{t-\varepsilon}^t \|I_1(r, \cdot)\|^2_{L^2_Q}\, {\rm d}r \right |^{\frac p2}
\le C_{N,p,Q,T} \varepsilon^{\frac p2-2} \varepsilon^{\frac p2}= C_{N,p,Q,T} \varepsilon^{p-2} .
\end{align*}
For the term $I_2$, by means of Fubini Theorem, H\"older's and Minkowski's inequalities and by \eqref{star2}, provided $p>4$, we get
\begin{align*}
\mathbb{E}&
\left | \int_{t-\varepsilon}^t \|I_2(r, \cdot)\|^2_{L^2_Q}\, {\rm d}r \right |^{\frac p2}\\
&= \mathbb{E}\left[ \int_{t-\varepsilon}^t\left\Vert \int_{t-\varepsilon}^t\int_D\left( \nabla_yg(t-s,x,y)\cdot \int_Dk(y-\alpha)D_{r,\cdot}\xi_N(s,\alpha)\, {\rm d}\alpha\right)\right. \right.
\\
&\left. \left. \qquad \qquad \qquad \Theta_N(\|\xi_N(s, \cdot)\|_{L^p})\xi_N(s,y)\, {\rm d}y \, {\rm d}s\right\Vert^2_{L^2_Q}\, {\rm d}r\right]^{\frac p2}
\\
&\le C_N \left( \int_{t-\varepsilon}^t (t-s)^{\frac{p+2}{2(1-p)}}\, {\rm d}s \right)^{p-1} \ \int_{t-\varepsilon}^t\int_D \mathbb{E} \|D\xi_N(s, y)\|^p_{\mathcal{H}_{(t-\varepsilon, t)}}\,  {\rm d}y\, {\rm d}s
\\
&\le C_N \varepsilon^{\frac p2 -2} C_{N,p,Q,T} \varepsilon^p=C_{N,p,Q,T} \varepsilon^{p-2} \qquad \text{by Lemma \ref{D_eps}}.
\end{align*}
As regards the last term $I_3$, Minkowski's and H\"older's inequalities and \eqref{stima_q_2} imply that 
 
\begin{align*}
\mathbb{E}&
\left | \int_{t-\varepsilon}^t \|I_3(r, \cdot)\|^2_{L^2_Q}\, {\rm d}r \right |^{\frac p2}
\\ 
 &=\mathbb{E} \left[
\int_{t - \varepsilon}^t 
\left\Vert p\int_{t-\varepsilon}^t \int_D \nabla_yg(t-s,x,y)\cdot \tilde q_N(\xi_N(s, \cdot))(y)\|\xi_N(s, \cdot)\|^{1-p}_{L^p} \right. \right.
\\
&\left. \left.\qquad \qquad  \left( \int_D |\xi_N(s, \beta)|^{p-2}\xi_N(s, \beta)D_{r, \cdot}\xi_N(s, \beta)\, {\rm d}\beta\right)\, {\rm d}y \, {\rm d}s
\right\Vert^2_{L^2_Q}\, {\rm d}r
\right]^{\frac p2}
\\
&\le C_{N,p}\left(\int_{t-\varepsilon}^t \int_D|\nabla_yg(t-s,x,y)|^{\frac{p}{p-1}}\, {\rm d}y \, {\rm d}s\right)^{p-1} 
\\
& \qquad \qquad  \int_{t-\varepsilon}^t\int_D \mathbb{E}\|D\xi_N(s,y)\|^p_{\mathcal{H}_{(t-\varepsilon, t)}}\, {\rm d}y \, {\rm d}s
\\
&\le C_{N,p} \varepsilon^{\frac p2-2} C_{N,p,Q,T} \varepsilon^p
=C_{N,p,Q,T} \varepsilon^{p-2}\qquad\text{ by Lemma \ref{D_eps}} \ \text{and \eqref{beta2} if } p>4.
\end{align*}
In conclusion, collecting all the above estimates, we get 

\begin{equation}
\label{I}
\mathbb{E}\left| I(t,x, \varepsilon)\right|^{\frac p2} \le C_{N,p,Q,T} \varepsilon^{p-2},
\end{equation}
provided $p>4$.
We now need to find a lower estimate for $A(x, \varepsilon)$. Proceeding as in \eqref{conv_stoc} we have
\begin{align*}
\int_{t- \varepsilon}^t \| g(t-r, x,\cdot) \|^2_{L^2_Q} \, {\rm d}r
= \sum_{k \in \mathbb{Z}^2_0}|k|^{-2b}|e_k(x)|^2 \ \frac{1}{2|k|^2}(1-e^{-2|k|^2 \varepsilon}).
\end{align*}
The inequality
\begin{align*}
1-e^{-2|k|^2 \varepsilon} \ge \frac{2 \varepsilon |k|^2}{1+2 \varepsilon |k|^2} \ge \frac{2 \varepsilon |k|^2}{1+2T |k|^2} 
\end{align*}
implies that
\begin{align*}
\int_{t- \varepsilon}^t \| g(t-r, x,\cdot) \|^2_{L^2_Q} \, {\rm d}r& 
\ge \frac{\varepsilon}{(2\pi)^2}\sum_{k \in \mathbb{Z}^2_0}\frac{|k|^{-2b}}{1+2T|k|^2}
\end{align*}
and the above series is well defined and can be bounded from below by any of its summand, such as the one corresponding to $k=(0,1)\in \mathbb{Z}_0^2$: 
\begin{align}
\label{A}
\int_{t- \varepsilon}^t \| g(t-r, x,\cdot) \|^2_{L^2_Q} \, {\rm d}r& \ge \frac{\varepsilon}{(2\pi)^2(1+2T)}= C_{T} \ \varepsilon.
\end{align}
Using estimates \eqref{I} and \eqref{A} and substituting into \eqref{Chebyc} we get
\begin{equation*}
\mathbb{P}(\|D\xi_N(t,x)\|^2_{\mathcal{H}_T}< \delta) \le \left( \frac{C_T}{2} \varepsilon - \delta\right)^{-\frac p2} C_{N,p,Q,T}\ \varepsilon^{p-2}. \end{equation*}
Thus, if we choose $\varepsilon= \varepsilon( \delta, T)$ sufficiently small in such a way that  $\frac{C_T}{2}\varepsilon =2 \delta$ we get
\begin{equation*}
\mathbb{P}(\|D\xi_N(t,x)\|^2_{\mathcal{H}_T}< \delta) \le C_{N,T,Q,p} 
\delta^{-\frac p2}\delta^{p-2}= C_{N,T,Q,p} \ \delta^{\frac p2-2} \rightarrow 0 \qquad \text{for} \ \delta \rightarrow 0,
\end{equation*}
since $p>4$.
\end{proof}

\subsection{Existence of the density}
\label{density_subsec}
Now we are ready to prove the main result, Theorem \ref{density_thm}.

\begin{proof}[Proof of Theorem \ref{density_thm}]
Let us fix $N \ge 1$ and $p>4$ and let us define
\begin{equation}
\Omega_N := \left\{\omega \in \Omega: \sup_{t \in \left[0,T\right]}\|\xi(t, \cdot, \omega)\|_{L^{p}(D)} \le N \right\}.
\end{equation}
Then we have $\xi(t,x) \equiv \xi_N(t,x)$ on $\Omega_N$ for every $t\in \left[0,T\right]$ and $x \in D$ and $\lim_{N \rightarrow +\infty} \mathbb{P}(\Omega_N=\Omega) =1$.
In fact we can write
$$\Omega_N= \left\{ \sigma_N=T\right\},$$
where $\sigma_N$ is the stopping time defined in \eqref{sigma_N}.
So we have that, for $N \rightarrow \infty$, $\sup_{N\ge1} \sigma_N =T$ $\mathbb{P}$-a.s. i.e. $\Omega_N \uparrow \Omega$ $\mathbb{P}$-a.s.

It follows then that, for every $(t,x) \in \left[0,T\right] \times D)$, the sequence $(\Omega_N, \xi_N(t,x))$ localizes $\xi(t,x)$ in $\mathbb{D}^{1,p}$.
The result follows by Theorem \ref{dens}: in fact it suffices to show property \eqref{44} on the set $\{t < \sigma_N\}$ for every $N \ge 1$, namely to show \eqref{58}.
\end{proof}

\appendix
\section{Proof of Theorem \ref{beta_gradient}}
\label{Appendix}
For the estimate of the heat kernel and its gradient we use the explicit expression given by \eqref {images}. We factorize the two-dimensional kernel into the one dimensional components. We then proceed following the idea of \cite[Lemma 2.1]{Morien1999}.

\begin{align*}
g(t,x,y)&=\frac{1}{4 \pi t}  \sum_{k \in \mathbb{Z}^2}e^{-\frac{|x-y+2 \pi k|^2}{4t}}
\\
&=\left(\frac{1}{\sqrt{4 \pi t}} \sum_{k_1\in \mathbb{Z}}e^{\frac{-|x_1-y_1+2 \pi k_1|^2}{4t}} \right)\left( \frac{1}{\sqrt{4 \pi t}} \sum_{k_2\in \mathbb{Z}}e^{\frac{-|x_2-y_2+2 \pi k_2|^2}{4t}}\right).
\end{align*}
Let us set, for $i=1,2$
\begin{equation*}
g_i(t,x_i,y_i)=\frac{1}{\sqrt{4 \pi t}} \sum_{k_i\in \mathbb{Z}}e^{\frac{-|x_i-y_i+2 \pi k_i|^2}{4t}},
\end{equation*}
then 
\begin{equation*}
g(t,x,y)=g_1(t,x_1, y_1)g_2(t,x_2,y_2).
\end{equation*}
For the one-dimensional heat kernel the following decomposition holds:
\begin{equation*}
g_i(t,x_i,y_i)=H^1_i(t, x_i,y_i)+H^2_i(t, x_i,y_i)+H^3_i(t, x_i,y_i)+\bar g_i(t, x_i,y_i)
\end{equation*}
where
 \begin{align*}
 H^1_i(t, x_i,y_i)=\frac{1}{\sqrt{4 \pi t}} & e^{\frac{-|x_i-y_i|^2}{4t}} ,
  \quad 
 H^2_i(t, x_i,y_i)=\frac{1}{\sqrt{4 \pi t}} e^{\frac{-|x_i-y_i+2 \pi |^2}{4t}} , 
 \\
 &H^3_i(t, x_i,y_i)=\frac{1}{\sqrt{4 \pi t}} e^{\frac{-|x_i-y_i-2 \pi|^2}{4t}}
\end{align*}
and 
\begin{equation}
\label{smooth}
(t, x_i,y_i) \rightarrow \bar g_i(t,x_i,y_i) \in C^{\infty}(\left[0,T\right] \times \mathbb{R}^2).
\end{equation}
Then we can rewrite the two dimensional heat kernel as follows
\begin{multline*}
g(t,x,y)=\left(H^1_1(t, x_1,y_1)+H^2_1(t, x_1,y_1)+H^3_1(t, x_1,y_1)+\bar g_1(t, x_1,y_1) \right) \cdot 
\\
\left(H^1_2(t, x_2,y_2)+H^2_2(t, x_2,y_2)+H^3_2(t, x_2,y_2)+\bar g_2(t, x_2,y_2) \right).
\end{multline*}
We are interested in estimating the heat kernel and its gradient, more precisely in estimates of the following type:
\begin{equation}
\label{integrals}
\int_0^t\int_{D}| g(s,x,y)|^{\beta} {\rm d}y {\rm d}s,
\qquad \qquad 
\int_0^t\int_{D}| \nabla_yg(s,x,y)|^{\beta} {\rm d}y {\rm d}s,
\end{equation}
for $t>0$ and a suitable $\beta>0$.

\begin{remark}
\label{rem_H}
Let us notice that the terms of the form $H_1^k\bar g_2$  and $H_2^k\bar g_1$ with $k=1,2,3$ do not give any problems. 
In fact let us consider for example the case $H^1_1 \bar g_2$ (the others are similar).
We have
\begin{align*}
&|\nabla_y  (H_1^1\bar g_2)|^{\beta}
=\left(|\nabla_y (H_1^1\bar g_2)|^2\right)^{\frac{\beta}{2}}
\\
&\le C_{\beta} \left(\frac{(2|x_1-y_1|)^{\beta}}{\pi^{\frac{\beta}{2}}(4 t)^{\frac{3\beta}{2}}} e^{-\frac{\beta|x_1-y_1|^2}{4t}}\left|\bar g_2(t, x_2, y_2)\right|^{\beta} 
+\frac{1}{(4 \pi t)^{\frac{\beta}{2}}} e^{-\frac{\beta|x_1-y_1|^2}{4t}}\left|\frac{\partial}{\partial y_2}\bar g_2(t, x_2, y_2)\right|^{\beta}\right)
\\
&\le C_{\beta}\frac{|x_1-y_1|^{\beta}}{t^{\frac{3\beta}{2}}} e^{-\frac{\beta|x_1-y_1|^2}{4t}}\left|\bar g_2(t, x_2, y_2)\right|^{\beta} 
+\frac{C_{\beta}}{t^{\frac{\beta}{2}}} e^{-\frac{\beta|x_1-y_1|^2}{4t}}\left|\frac{\partial}{\partial y_2}\bar g_2(t, x_2, y_2)\right|^{\beta}.
\end{align*}
Then, using the following identity
\begin{equation}
\label{gaussiana}
\int_{\mathbb{R}}|z|^r e^{-\frac{z^2}{\sigma^2}} \, {\rm d}z = C_r \sigma^{r+1}
\end{equation}
we get
\begin{align*}
\int_0^t\int_D&|\nabla_y(H^1_1\bar g_2)(s,x,y)|^{\beta} {\rm d}y \,{\rm d}s 
\\
&\le C_{\beta}\int_0^t \int_0^{2 \pi} \left(\int_0^{2 \pi}\frac{|x_1-y_1|^{\beta}}{s^{\frac{3\beta}{2}}} e^{-\frac{\beta|x_1-y_1|^2}{4s}}\, {\rm d}y_1\right)\left|\bar g_2(s, x_2, y_2)\right|^{\beta} {\rm d}y_2 {\rm d}s
\\
&+C_{\beta}\int_0^t \frac{1}{s^{\frac {\beta}{2}}}\int_0^{2 \pi}\left( \int_0^{2\pi}e^{-\frac{\beta|x_1-y_1|^2}{4s}}\, {\rm d}y_1\right)
\left|\frac{\partial}{\partial y_2} \bar g_2(s,x_2,y_2)\right|^{\beta}\, {\rm d}y_2 \, {\rm d}s
\\
&\le C_{\beta}\int_0^t \int_0^{2 \pi} s^{\frac12- \beta} \ \left|\bar g_2(s, x_2, y_2)\right|^{\beta}\, {\rm d}y_2\, {\rm d }s 
\\
&\qquad \qquad + C_{\beta}\int_0^t \int_0^{2 \pi} s^{\frac{1- \beta}{2}}\left|\frac{\partial}{\partial y_2}\bar g_2(s, x_2, y_2) \right|^{\beta}\,{\rm d}y_2 {\rm d}s
\end{align*}
and we have the convergence of the integrals thanks to \eqref{smooth}, when $\beta < \frac 32$.
\end{remark}

By Remark \ref{rem_H} it follows that the behavior of integrals in \eqref{integrals} is determined by the corresponding integrals with $H_1^k H_2^l$ with $k,l =1,2,3$, instead of $g$.
Since computations are similar we do all the required estimates only for the case $H(t,x,y):=H^1_1(t,x_1,y_1)H^1_2(t,x_2,y_2)$. We have
\begin{equation*}
|\nabla_yH(t,x,y)|^{\beta}= \frac{e^{- \frac{\beta|x-y|^2}{4t}}|x-y|^{\beta}}{(8 \pi)^{\beta}t^{2\beta}},
\end{equation*}
so we recover
\begin{align*}
 \int_D&|\nabla_yH(s,x,y)|^{\beta}{\rm d}y
=\int_D\frac{e^{- \frac{\beta|x-y|^2}{4s}}|x-y|^{\beta}}{(8 \pi)^{\beta}s^{2\beta}}{\rm d}y 
\\
& \le C_{\beta}\int_{\mathbb{R}^2}\frac{e^{- \frac{\beta|z|^2}{4s}}|z|^{\beta}}{s^{2\beta}}{\rm d}z 
=  C_{\beta}\int_0^{2 \pi}\int_0^{\infty}\frac{e^{- \frac{\beta\rho^2}{4s}}\rho^{\beta+1}}{s^{2\beta}}\,{\rm d}\rho\,{\rm d} \phi
\\
& \le C_{\beta} \frac{1}{s^{2 \beta}} \int_0^{\infty}\rho^{\beta+1}e^{-\frac{\beta \rho^2}{4s}} \, {\rm d} \rho.
\end{align*}
Using now identity \eqref{gaussiana} we get
\begin{align*}
 \int_D&|\nabla_yH(s,x,y)|^{\beta}{\rm d}y \,{\rm d}s
\le C_{\beta}  s^{-\frac{3 \beta}{2}+1}.
\end{align*}
Calculating the time integral we obtain, 
\begin{align}
\label{09_12}
\int_0^t \int_D&|\nabla_yH(s,x,y)|^{\beta}{\rm d}y \,{\rm d}s
\le C_{\beta} \int_0^t s^{-\frac{3 \beta}{2}+1} {\rm d}s
\le C_{\beta}t^{-\frac {3\beta}{2}+2},
\end{align}
which converges provided $ \beta < \frac 43$.
\begin{remark}
Notice that estimate \eqref{09_12} is uniform in $x$.
\end{remark}
For estimates \eqref{beta_kernel_0} and \eqref{beta_kernel} we proceed in a similar way. 
Also in this case we do all the required estimates for $H(t,x,y):=H^1_1(t,x_1,y_1)H^1_2(t,x_2,y_2)$. By means of \eqref{gaussiana} we get
\begin{align*}
 \int_D |H(s,x,y)|^{\beta}\,{\rm d}y
& = \int_D\frac{1}{(4 \pi s)^{\beta}}e^{-\frac{\beta|x-y|^2}{4s}}\, {\rm d}y
\le C_{\beta} \int_{\mathbb{R}^2}\frac{1}{s^{\beta}}e^{-\frac{\beta|z|^2}{4s}}\, {\rm d}z
\\
& = 2 \pi C_{\beta} \int_0^{\infty}\frac{e^{-\frac{\beta\rho^2}{4s}}}{s^{\beta}}\rho\, {\rm d}\rho
\le C_{\beta}s^{1-\beta}.
\end{align*}
Computing the time integral we obtain
\begin{align*}
\int_0^t \int_D |H(s,x,y)|^{\beta}\,{\rm d}y\, {\rm d}s
\le C_{\beta}\int_0^t s^{1-\beta}\, {\rm d}s \le C_{\beta}t^{2-\beta},
\end{align*}
which converges provided $\beta<2$.


\end{document}